\RequirePackage{amsmath}
\documentclass[runningheads]{llncs}
\usepackage[T1]{fontenc}
\usepackage{graphicx}
\usepackage{hyperref}
\usepackage{color}

\usepackage[dvipsnames, table]{xcolor}
\usepackage{amsmath,amssymb}
\usepackage{wrapfig}
\usepackage[capitalise]{cleveref}
\usepackage{hyperref}
\usepackage{thmtools}
\usepackage{thm-restate}
\usepackage{enumitem}
\usepackage{mathrsfs}  
\usepackage[all]{xy}
\usepackage{multirow}
\usepackage{caption} 
\usepackage{float}
\usepackage{multicol}
\usepackage{soul}
\usepackage{hhline}
\usepackage{bbm}
\usepackage{mathtools}
\usepackage{bm}
\usepackage{ dsfont }
\usepackage{pifont}
\usepackage{array}
\usepackage{pdflscape}
\usepackage{mdframed}
\usepackage{tikz-cd}
\usepackage{tikz}
\usepackage{pgfplots}
    \pgfplotsset{compat=newest}
\usepackage[ttdefault=true]{AnonymousPro}
\usepackage{adjustbox}
\usepackage{makecell}
\usepackage{tabstackengine}
\usepackage{xspace}
\usepackage{soul}
\usepackage[pass]{geometry}
\usepackage{wrapfig}

\usepackage[misc]{ifsym}
\usepackage{orcidlink}

\tikzset{
    rotated halfcircle/.style={%
        mark=halfcircle*,
        mark color=black,
        fill=red,
        every mark/.append style={rotate=#1}
    }
}

\newcommand{\s}{\ensuremath{\sigma}}
\newcommand{\shA}{\ensuremath{\hat{\sA}}}

\renewcommand{\S}{\ensuremath{\mathcal{S}}}
\renewcommand{\P}{\ensuremath{\mathcal{P}}}

\newcommand{\wit}{\textit{wit}}
\newcommand{\overarrow}{\overrightarrow}

\newcommand{\Exists}[1]{\exists\,#1.\:}
\newcommand{\Forall}[1]{\forall\,#1.\:}
\newcommand{\Q}[2]{Q_{#1}\,#2.}

\newcommand{\Saa}[1]{\S^{\A}_{#1\aleph_{0}}}
\newcommand{\Sasa}[1]{\S^{\sA}_{#1\aleph_{0}}}

\renewcommand{\t}{\tau}
\newcommand{\Sigmadag}{\Sigma^{\dag}}
\newcommand{\sAdag}{\sA^{\dag}}
\newcommand{\sBdag}{\sB^{\dag}}
\newcommand{\For}[1]{For(#1)}

\newcommand{\GT}[1]{{}}
\newcommand{\A}{\ensuremath{\mathcal{A}}}
\newcommand{\B}{\ensuremath{\mathcal{B}}}

\newcommand{\sA}{\ensuremath{\mathbb{A}}}
\newcommand{\sB}{\ensuremath{\mathbb{B}}}
\newcommand{\sC}{\ensuremath{\mathbb{C}}}

\newcommand{\F}{\ensuremath{\mathcal{F}}}
\newcommand{\T}{\ensuremath{\mathcal{T}}}
\renewcommand{\S}{\ensuremath{\mathcal{S}}}
\renewcommand{\P}{\ensuremath{\mathcal{P}}}
\newcommand{\vars}{\textit{vars}}

\renewcommand{\t}{\tau}
\newcommand{\nset}[2]{\ensuremath{P_{#2}(#1)}}
\newcommand{\minmods}[2]{\ensuremath{\mathrm{minmods}_{#1, #2}}}
\newcolumntype{P}[1]{>{\centering\arraybackslash}p{#1}}

\newcommand{\inj}{\ensuremath\mathit{inj}\xspace}
\newcommand{\sur}{\ensuremath\mathit{sur}\xspace}

\Crefname{theorem}{Theorem}{Theorems}
\Crefname{corollary}{Corollary}{Corollaries}
\Crefname{example}{Example}{Examples}

\begin{document}
%\linenumbers
%
\title{The nonexistence of unicorns and many-sorted L{\"o}wenheim--Skolem theorems}
\author{Benjamin Przybocki\inst{1}\textsuperscript{(\Letter)} \orcidlink{0009-0007-5489-1733} 
\and
Guilherme Toledo\inst{2} \orcidlink{0000-0002-6539-398X} 
\and
Yoni Zohar\inst{2} \orcidlink{0000-0002-2972-6695} 
\and Clark Barrett\inst{1} \orcidlink{0000-0002-9522-3084}
}
\authorrunning{B. Przybocki et al.}
\institute{Stanford University, USA\\
\email{benjamin.przybocki@gmail.com, barrett@cs.stanford.edu} \and
Bar-Ilan University, Israel\\
\email{\{guivtoledo,yoni206\}@gmail.com}}
\maketitle              % typeset the header of the contribution
\begin{abstract}
Stable infiniteness, strong finite witnessability, and smoothness are model-theoretic properties relevant to theory combination in satisfiability modulo theories. Theories that are strongly finitely witnessable and smooth are called \emph{strongly polite} and can be effectively combined with other theories. Toledo, Zohar, and Barrett conjectured that stably infinite and strongly finitely witnessable theories are smooth and therefore strongly polite. They called counterexamples to this conjecture \emph{unicorn theories}, as their existence seemed unlikely. We prove that, indeed, unicorns do not exist. We also prove versions of the L{\"o}wenheim--Skolem theorem and the {\L}o\'s--Vaught test for many-sorted logic.
%\keywords{satisfiability modulo theories  \and many-sorted logic \and theory combination \and L{\"o}wenheim--Skolem theorem.}
\end{abstract}
\section{Introduction}

Given decision procedures for theories $\T_1$ and $\T_2$
with disjoint signatures, 
is there a decision procedure for $\T_1 \cup \T_2$?
In general, the answer is ``not necessarily'', 
 but a
 central question in Satisfiability Modulo Theories (SMT)~\cite{BT18} is: what assumptions on $\T_1$ and $\T_2$ 
suffice for
theory combination? 
This line of research began with Nelson and Oppen's theory combination procedure \cite{NelsonOppen}, which applies when $\T_1$ and $\T_2$ are stably infinite, 
roughly meaning that every $\T_i$-satisfiable quantifier-free formula is satisfied by an infinite $\T_i$-interpretation for $i \in \{1,2\}$.

The Nelson--Oppen procedure is quite useful, but requires \emph{both} theories to be stably infinite,
which is not always the case (e.g., the theories of bit-vectors and finite datatypes are not stably infinite).
Thus, sufficient properties of 
only one of the theories were identified,
%that enable theory combination with limited assumptions on the other,
such as
gentleness \cite{gentle}, shininess \cite{shiny}, and flexibility~\cite{flexible}. 
The most relevant property for our purposes is 
strong politeness~\cite{polite,stronglypolite,casalrasga2018,algebraicdatatypes}. It is essential to the functioning of the SMT solver cvc5 \cite{cvc5}, which is called billions of times per day in industrial production code. %Actually, we only consider strong politeness \cite{stronglypolite}, a strengthening which fixes a slight defect in the definition of politeness.\footnote{The name ``strong politeness'' is due to~\cite{casalrasga2018}.} 
A theory is \emph{strongly polite} if it is smooth and strongly finitely witnessable, which are model-theoretic properties we will define later. 
These properties are more involved than stable infiniteness, 
so proving a theory to be strongly polite is more difficult. % than proving it to be stably infinite. 
But the advantage of strongly polite theories is that they can be combined with any other decidable theory, 
including theories that are not stably infinite.

Given the abundance of model-theoretic properties relevant to theory combination, some of which interact in subtle ways, it behooves us to understand the logical relations between them. 
Recent papers \cite{BarTolZoh,BarTolZoh2} 
have sought to understand the relations between seven model-theoretic properties---including stable infiniteness, smoothness, and strong finite witnessability---by determining which combinations of properties are possible in various signatures. 
In most cases, a theory with the desired combination of properties was constructed,
 or it was proved that none exists. 
 The sole exception was theories that are stably infinite and strongly finitely witnessable but not smooth, dubbed \emph{unicorn theories} and conjectured not to exist. Our main result, \Cref{thm-no-unicorn}, confirms this conjecture.

Besides completing the taxonomy of properties from~\cite{BarTolZoh,BarTolZoh2}, 
our result has practical consequences. The nonexistence of unicorns implies that strongly polite theories can be equivalently defined as those that are stably infinite and strongly finitely witnessable. Since it is easier to prove that a theory is stably infinite than to prove that it is smooth, this streamlines the process of proving that a theory is strongly polite.
Thus, each time a new theory is introduced, proving that it can be combined with other theories becomes easier.\footnote{\cite{BarTolZoh} already proved that stably infinite and strongly finitely witnessable theories can be combined with other theories. Our result gives a new proof (see \Cref{cor-polite}), and shows that their procedure is not more general than polite combination.} Similarly, our results give a new characterization of shiny theories, which makes it easier to prove that a theory is amenable to the shiny combination procedure (see \Cref{cor-shiny}).

We also believe that our result is of theoretical interest. \Cref{thm-smoothness-from-finite-smoothness}, which is the main ingredient in the proof of \Cref{thm-no-unicorn}, can be seen as a variant of the upward L{\"o}wenheim--Skolem theorem for many-sorted logic, since proving that a theory is smooth amounts to proving that cardinalities of sorts can be increased arbitrarily, including to uncountable cardinals. This result may be of independent interest to logicians studying the model theory of many-sorted logic, and we hope the proof techniques are useful to them as well.

Speaking of proof techniques, our proof is curious in that it uses Ramsey's theorem from finite combinatorics. This is not the first time Ramsey's theorem has been used in logic. Ramsey proved his theorem in the course of solving a special case of the decision problem for first-order logic \cite{ramsey1929}. Ramsey's theorem also shows up in the Ehrenfeucht--Mostowski construction in model theory \cite{ehrenfeuchtmostowski}. Our proof actually requires a generalization of Ramsey's theorem, which we prove using the standard version of Ramsey's theorem.

A major component of the proof of \Cref{thm-no-unicorn} amounts to proving a many-sorted version of the L{\"o}wenheim--Skolem theorem. On the course to proving this, we realized that a proper understanding of this theorem for many-sorted logic appears to be missing from the literature,
despite the fact that the SMT-LIB standard~\cite{BarFT-RR-17}
is based on many-sorted logic.
To fill this gap, we prove generalizations of the L{\"o}wenheim--Skolem theorem for many-sorted logic, and use them to prove a many-sorted {\L}o\'s--Vaught test, useful for proving theory completeness.

The remainder of this paper is structured as follows.
\Cref{preliminaries} provides background and definitions
on many-sorted logic and SMT.
\Cref{NoUnicorns} %uses these new theorems in order to 
proves the main
result of this paper, namely the nonexistence of unicorn theories.
\Cref{LowenheimSkolem} proves new many-sorted variants of
the L{\"o}wenheim--Skolem theorem.
\Cref{conclusion} concludes and presents directions for future work.\footnote{Some proofs are omitted. They can be found in the appendix.}

\section{Preliminaries}\label{preliminaries}

\subsection{Many-sorted first-order logic}

We work in many-sorted first-order logic \cite{Monzano93}. A \emph{signature} $\Sigma$ consists of a non\-empty set $\S_{\Sigma}$ of sorts, a set $\F_{\Sigma}$ of function symbols, and a set $\P_{\Sigma}$ of predicate symbols containing an equality symbol $=_{\s}$ for every sort $\s \in \S_{\Sigma}$.\footnote{When specifying a signature, we often omit the equality symbols, and include them implicitly. We also omit $\sigma$ from $=_{\sigma}$ when it does not cause confusion.} Every function symbol has an arity $(\s_1, \dots, \s_n, \s)$ and every predicate symbol an arity $(\s_1, \dots, \s_n)$, where $\s_1, \ldots, \s_n, \s \in \S_{\Sigma}$ and $n\ge 0$.
Every equality symbol $=_{\s}$ has arity $(\s, \s)$. To quantify a variable $x$ of sort $\s$, we write $\Forall{x : \s}$ and $\Exists{x : \s}$ for the universal and existential quantifiers respectively. Let $|\Sigma| = |\S_{\Sigma}|+|\F_{\Sigma}|+|\P_{\Sigma}|$. If a signature contains only sorts and equalities, we say it is \emph{empty}. Two signatures are said to be \emph{disjoint} if they share at most sorts and equality symbols.

We define $\Sigma$-terms and $\Sigma$-formulas as usual.
%Further, the sort of a $\Sigma$-term is defined by induction in the usual way.
The set of free variables of sort $\s$ in  $\varphi$ is denoted $\vars_{\s}(\varphi)$. 
For $S\subseteq\S_{\Sigma}$, let $\vars_S(\varphi) = \bigcup_{\s \in S} \vars_{\s}(\varphi)$. 
We also let $\vars(\varphi) = \vars_{\S_{\Sigma}}(\varphi)$.
A $\Sigma$-sentence is a $\Sigma$-formula with no free variables.

A \emph{$\Sigma$-structure} $\sA$ interprets each sort $\s \in \S_{\Sigma}$ as a nonempty set $\s^\sA$, each function symbol $f \in \F_{\Sigma}$ as a function $f^\sA$ with the appropriate domain and codomain, and each predicate symbol $P \in \P_{\Sigma}$ as a relation $P^\sA$ over the appropriate set, such that $=_{\s}^{\sA}$ is the identity on $\s^\sA$. 
A \emph{$\Sigma$-interpretation} $\A$ is a pair $(\sA, \nu)$, where $\sA$ is a $\Sigma$-structure and $\nu$ is a function, called an \emph{assignment}, mapping each variable $x$ of sort $\s$ to an element $\nu(x) \in \s^\sA$, denoted $x^\A$. We write $t^{\A}$ for the interpretation of the $\Sigma$-term $t$ under $\A$,
which is defined in the usual way. The entailment relation, denoted $\vDash$, is defined as usual.

Two structures are \emph{elementarily equivalent} if they satisfy the same sentences. We say that $\sA$ is an \emph{elementary substructure} of $\sB$ if $\sA$ is a substructure of $\sB$ and, for all formulas $\varphi$ and all assignments $\nu$ on $\sA$, we have $(\sA,\nu) \vDash \varphi$ if and only if $(\sB, \nu) \vDash \varphi$. Note that if $\sA$ is an elementary substructure of $\sB$, then they are elementarily equivalent.  $\A$ is an elementary subinterpretation of $\B$ if $\sA$ is an elementary substructure of $\sB$ and $\A$'s assignment is the same as $\B$'s assignment.

Given a $\Sigma$-structure $\sA$, let $\Sasa{\geq}=\{\s\in \S_{\Sigma}: |\s^{\sA}|\geq\aleph_{0}\}$ and $\Sasa{<}=\S_{\Sigma}\setminus \Sasa{\geq}$. We similarly define $\Saa{\geq}$ and $\Saa{<}$ for a $\Sigma$-interpretation $\A$.  
% We use the same notation, but replacing $S$ with $\S$ for the case when $S=\S_\Sigma$.

A \emph{$\Sigma$-theory} $\T$ is a set of $\Sigma$-sentences, called the \emph{axioms} of $\T$. We write $\vdash_\T \varphi$ instead of $\T \vDash \varphi$. Structures satisfying $\T$ are called $\T$-\emph{models}, and interpretations satisfying $\T$ are called $\T$-\emph{inter\-pre\-ta\-tions}. We say a $\Sigma$-formula is $\T$-\emph{satisfiable} if it is satisfied by some $\T$-interpretation, and we say two $\Sigma$-formulas are $\T$-\emph{equivalent} if every $\T$-interpretation satisfies one if and only if it satisfies the other.  
$\T$ is \emph{complete} if for every sentence $\varphi$, we have $\vdash_{\T} \varphi$ or $\vdash_{\T} \neg \varphi$. 
$\T$ is \emph{consistent} if there is no formula $\varphi$ such that $\vdash_{\T} \varphi$ and $\vdash_{\T}\neg\varphi$.
If $\Sigma_1$ and $\Sigma_2$ are disjoint, let $\Sigma_1 \cup \Sigma_2$ be the signature with the union of their sorts, function symbols, and predicate symbols. 
Given a $\Sigma_1$-theory $\T_1$ and a $\Sigma_2$-theory $\T_2$, the $(\Sigma_1 \cup \Sigma_2)$-theory $\T_1 \cup \T_2$ is the theory whose axioms are the union of the axioms of $\T_1$ and $\T_2$.

The following theorem, proved in \cite{Monzano93}, is a many-sorted variant of the first-order
compactness theorem.

\begin{theorem}[Compactness Theorem \cite{Monzano93}] \label{compact}
    A set of $\Sigma$-formulas $\Gamma$ is satisfiable if and only if every finite subset of $\Gamma$ is satisfiable.
\end{theorem}

We say that a $\Sigma$-theory $\T$ has \emph{built-in Skolem functions} if for all formulas $\psi(\overarrow{x}, y)$, there is $f \in \F_\Sigma$ such that $\vdash_{\T} \Forall{\overarrow{x}} (\Exists{y} (\psi(\overarrow{x}, y)) \rightarrow \psi(\overarrow{x}, f(\overarrow{x})))$.\footnote{Intuitively: $\T$ has
enough function symbols to witness all existential formulas.}
The following is a many-sorted variant of 
Lemma~2.3.6 of
\cite{marker2002}.
The proof is almost identical to that of the single-sorted case
from~\cite{marker2002}.

\begin{restatable}{lemma}{builtin}\label{builtin}
    If $\T$ is a $\Sigma$-theory for a countable $\Sigma$, then there is a countable signature $\Sigma^* \supseteq \Sigma$ and $\Sigma^*$-theory $\T^* \supseteq \T$ with built-in Skolem functions.
\end{restatable}

% We now state two standard model-theoretic theorems that we make use of.

We state a many-sorted generalization of the Tarski--Vaught test, whose proof is also similar to the single-sorted case \cite[Proposition~2.3.5]{marker2002}.

\begin{restatable}[The Tarski--Vaught test]{lemma}{TVTEST}\label{TV test}
    Suppose $\sA$ is a substructure of $\sB$. Then, $\sA$ is an elementary substructure of $\sB$ if and only if $(\sB, \nu) \vDash \Exists{v} \varphi(\overarrow{x}, v)$ implies $(\sA, \nu) \vDash \Exists{v} \varphi(\overarrow{x}, v)$
    for every formula $\varphi(\overarrow{x}, v)$ and assignment $\nu$ over $\sA$.
\end{restatable}

\subsection{Model-theoretic properties}\label{Model-theoretic properties}

\begin{definition}
Let $\Sigma$ be a many-sorted signature,
$S \subseteq \S_{\Sigma}$,
and $\T$ a $\Sigma$-theory.
\begin{itemize}
\item    $\T$ is \emph{stably infinite} with respect to $S$ if for every $\T$-satisfiable quantifier-free formula $\varphi$, there is a $\T$-interpretation $\A$ satisfying $\varphi$ with $|\s^{\A}| \ge \aleph_0$ for every $\s \in S$.
\item    $\T$ is \emph{stably finite} with respect to $S$ if for every quantifier-free 
    $\Sigma$-formula $\varphi$ and $\T$-interpretation $\A$ satisfying $\varphi$, there is a $\T$-interpretation $\B$ satisfying $\varphi$ such that $|\s^\B| \le |\s^\A|$ and $|\s^\B| < \aleph_0$ for every $\s \in S$.
\item $\T$ is \emph{smooth} with respect to $S$ if for every quantifier-free formula $\varphi$, $\T$-interpretation $\A$ satisfying $\varphi$, and function $\kappa$ from $S$ to the class of cardinals such that $\kappa(\s) \geq |\s^{\A}|$ for every $\s \in S$, there is a $\T$-interpretation $\B$ satisfying $\varphi$ with $|\s^{\B}|=\kappa(\s)$ for every $\s \in S$.
\end{itemize}
\end{definition}

%\noindent
%Note that smoothness implies stable infiniteness.
%Note that if $\T$ is smooth, then $\T$ is stably infinite. % \cite[Theorem~1]{BarTolZoh}.
%
%We next define arrangements.

Next, we define
{\em arrangements}.
    Given a set of sorts $S \subseteq \S_{\Sigma}$, finite sets of variables $V_{\s}$ of sort $\s$ for each $\s\in S$, and equivalence relations $E_{\s}$ on $V_{\s}$, the \emph{arrangement} $\delta_V$ on $V=\bigcup_{\s\in S}V_{\s}$ induced by $E=\bigcup_{\s\in S}E_{\s}$ is 
    \[
        \bigwedge_{\s\in S}\left[\bigwedge_{xE_{\s}y}(x=y) \land \bigwedge_{x\overline{E_{\s}}y}\neg(x=y)\right],
    \]
    where $\overline{E_{\s}}$ is the complement of $E_{\s}$.

\begin{definition}
Let $\Sigma$ be a many-sorted signature,
$S \subseteq \S_{\Sigma}$ a finite set,
and $\T$ a $\Sigma$-theory.
Then $\T$ is \emph{strongly finitely witnessable} with respect to $S$ if there is a computable function $\wit$ from the quantifier-free formulas into themselves such that for every quantifier-free formula $\varphi$:
    \begin{enumerate}[label=(\roman*), itemsep=0pt]
        \item $\varphi$ and $\Exists{\overarrow{w}} \wit(\varphi)$ are $\T$-equivalent, where $\overarrow{w}=\vars(\wit(\varphi))\setminus\vars(\varphi)$; and
        \item given a finite set of variables $V$ and an arrangement $\delta_{V}$ on $V$, if $\wit(\varphi) \land \delta_{V}$ is $\T$-satisfiable, then there is a $\T$-interpretation $\A$ satisfying $\wit(\varphi) \land \delta_{V}$ such that $\s^{\A}=\vars_{\s}(\wit(\varphi) \land \delta_{V})^{\A}$ for every $\s\in S$.
    \end{enumerate}
\end{definition}

\subsection{Notation} \label{sec-notation}
$\mathbb{N}$ denotes the set of non-negative integers. Given $m,n \in \mathbb{N}$, let
$
    [m,n]  := \{\ell \in \mathbb{N} : m \le \ell \le n\} 
    $
    and
   $ [n]  := [1,n].
$
%
%
% \begin{align*}
%     [m,n] & := \{\ell \in \mathbb{N} : m \le \ell \le n\} \\
%     [n] & := [1,n].
% \end{align*}
Given a set $X$, let
$
    \nset{X}{n}  := \{Y \subseteq X : |Y| = n\}$,
    $X^n  := \{(x_1, \dots, x_n) : x_i \in X \ \text{for all} \ i \in [n]\}$, and
    $X^*  := \bigcup_{n \in \mathbb{N}} X^n$.
%
%
% \begin{align*}
%     \nset{X}{n} & := \{Y \subseteq X : |Y| = n\} \\
%     X^n & := \{(x_1, \dots, x_n) : x_i \in X \ \text{for all} \ i \in [n]\} \\
%     X^* & := \bigcup_{n \in \mathbb{N}} X^n.
% \end{align*}
%\noindent
For any $x$, we denote $(x,\ldots,x)$ by $(x)^{\oplus n}$.
% \[
%     (x)^n = (x, \dots, x) \in X^n.
% \]
Given a tuple of tuples $(\overarrow{x_1}, \dots, \overarrow{x_n})$, where $\overarrow{x_i} \in X^*$ for all $i$, we will often treat it as an element of $X^*$ by flattening the tuple.

\section{The nonexistence of unicorns}\label{NoUnicorns}

We now state our main theorem, which implies that unicorn theories do not exist. 
Note that since we are motivated by applications to SMT, we hereafter assume all signatures are countable.\footnote{The paper that introduced unicorn theories \cite{BarTolZoh} also made this assumption.}

\begin{theorem} \label{thm-no-unicorn}
    Assume that $\T$ is a $\Sigma$-theory, where $\Sigma$ is countable. If $\T$ is stably infinite and strongly finitely witnessable, both with respect to $S \subseteq \S_\Sigma$, then $\T$ is smooth with respect to $S$.
\end{theorem}

For our proof, we define a weaker variant of smoothness, that focuses the requirements only for finite cardinals.

\begin{definition}\label{finite smoothness}
    A $\Sigma$-theory $\T$ is \emph{finitely smooth} with respect to $S \subseteq \S_{\Sigma}$ if for every quantifier-free formula $\varphi$, $\T$-interpretation $\A$ with $\A\vDash\varphi$, and function $\kappa$ from $\Saa{<}\cap S$ to the class of cardinals with $|\s^{\A}| \leq \kappa(\s) < \aleph_0$ for every $\s \in \Saa{<}\cap S$, there is a $\T$-interpretation $\B$ with $\B\vDash\varphi$ with $|\s^{\B}|=\kappa(\s)$ for every $\s \in \Saa{<}\cap S$.
\end{definition}

We make use of the following two lemmas.

\begin{restatable}{lemma}{lemseventhree}\label{lem-73}
    If $\T$ is stably infinite and strongly finitely witnessable, both with respect to some set of sorts $S \subseteq \S_\Sigma$, then $\T$ is finitely smooth with respect to $S$.
\end{restatable}

\begin{lemma}[{\cite[Theorem~3]{BarTolZoh2}}] \label{lem-fin-model}
    If $\T$ is strongly finitely witnessable with respect to some set of sorts $S \subseteq \S_\Sigma$, then $\T$ is stably finite with respect to $S$.
\end{lemma}
% \begin{proof}
%     Let $\T$ be a strongly finitely witnessable theory with respect to $S$. Let $\varphi$ be a $\T$-satisfiable quantifier-free formula and $\A$ a $\T$-interpretation satisfying $\varphi$. We have $\A \vDash \Exists{\overarrow{w}} \wit(\varphi)$, where $\overarrow{w}=\vars(\wit(\varphi))\setminus\vars(\varphi)$. Hence, by modifying the interpretation of the variables in $\overarrow{w}$, we obtain a $\T$-interpretation $\A'$ satisfying $\wit(\varphi)$. Let $V = \vars(\wit(\varphi))$, and let $\delta_V$ be the arrangement on $V$ induced by the equalities in $\A'$. Then, $\wit(\varphi) \land \delta_V$ is $\T$-satisfiable, so there exists a $\T$-interpretation $\B$ satisfying $\wit(\varphi) \land \delta_V$ such that $\s^{\B}=\vars_{\s}(\wit(\varphi) \land \delta_V)^{\B}$ for every $\s \in S$. In particular, $|\s^\B| < \aleph_0$ for every $\s \in S$. Since $\B$ satisfies $\wit(\varphi)$, it also satisfies $\varphi$. Finally, the map from $\s^{\B}$ to $\s^{\A'}$ given by $x^\B \mapsto x^{\A'}$, where $x \in \vars_\s(\wit(\varphi))$, is well-defined and injective, so we have $|\s^{\B}| \le |\s^{\A'}| = |\s^{\A}|$ for every $\s \in S$. Therefore, $\T$ is stably finite with respect to $S$.
% \end{proof}

In light of the above two lemmas, the following theorem implies \Cref{thm-no-unicorn}.

\begin{theorem} \label{thm-smoothness-from-finite-smoothness}
    Assume that $\T$ is a $\Sigma$-theory, where $\Sigma$ is countable. If $\T$ is stably finite and finitely smooth, both with respect to some set of sorts $S \subseteq \S_\Sigma$, then $\T$ is smooth with respect to $S$.
\end{theorem}

% If one were to diagrammatically represent \Cref{thm-smoothness-from-finite-smoothness}, the result would again be \Cref{fig:diagrams}.(h); after all, smoothness implies we may take any cardinal for the domain of a sort larger than that of the initial structure, independently of the size of other sorts' domains. At the same time, we could also represent \Cref{thm-no-unicorn} and \Cref{lem-73} by means of the Euler diagram in \Cref{venndiagram}. While the former result states that the conjunction of stable infiniteness (which we denote in \Cref{venndiagram} by $\stainf$) and strong finite witnessability ($\strfinwit$) implies smoothness ($\smooth$), the latter states stable infiniteness and strong finite witnessability imply finite smoothness ($\finsmo$).
% Also, from the definitions, smoothness implies finite smoothness,
% which in turn implies stable infiniteness.

% Investigating
% whether \Cref{thm-smoothness-from-finite-smoothness} can be extended to uncountable signatures is left for future work. We now turn to describing the proof of \Cref{thm-smoothness-from-finite-smoothness}.

The remainder of this section is thus dedicated to the proof of \Cref{thm-smoothness-from-finite-smoothness}.

\subsection{Motivating the proof}

 In this section, we illustrate the proof technique with a simple example. The goal is to motivate the proof of \Cref{thm-smoothness-from-finite-smoothness} before delving into the details.

Suppose $\T$ is a $\Sigma$-theory, where $\S_{\Sigma} = \{\s_1, \s_2\}$, $\F_{\Sigma} = \{f\}$, $f$ has arity $(\s_2, \s_1)$, and the only predicate symbols are equalities. Suppose that $\T$ is also stably finite and finitely smooth, both with respect to $S = \S_{\Sigma}$. Let $\varphi$ be a $\T$-satisfiable quantifier-free formula and $\A$ a $\T$-interpretation satisfying $\varphi$. Let $\kappa$ be a function from $S$ to the class of cardinals such that $\kappa(\s) \ge |\s^\A|$ for both $\s \in S$. For concreteness, suppose $|\s_1^\A| = |\s_2^\A| = 10$, $\kappa(\s_1) = \aleph_0$, and $\kappa(\s_2) = \aleph_1$. Our goal is to show that there is a $\T$-interpretation $\B^-$ satisfying $\varphi$ with $|\s_1^{\B^-}| = \aleph_0$ and $|\s_2^{\B^-}| = \aleph_1$.\footnote{The reason for the $-$ superscript in $\B^-$ will be clear presently.}

A natural thought is to apply some variant of the upward L{\"o}wenheim--Skolem theorem, but this doesn't quite work. As will be seen in \Cref{LowenheimSkolem}, generalizations of the L{\"o}wenheim--Skolem theorem to many-sorted logic do not let us control the cardinalities of $\s_1$ and $\s_2$ independently. Nevertheless, let us emulate the standard proof technique for the upward L{\"o}wenheim--Skolem theorem.

Here is the most natural way of generalizing the proof of the upward L{\"o}wenheim--Skolem theorem to our setting. For simplicity, assume that $\T$ already has built-in Skolem functions. We introduce $\aleph_0$ new constants $\{c_{1,\alpha}\}_{\alpha < \omega}$ and $\aleph_1$ new constants $\{c_{2,\alpha}\}_{\alpha < \omega_1}$. We define a set of formulas $\Gamma = \{\varphi\} \cup \Gamma_1$, where
\[
    \Gamma_1 = \{\neg (c_{i,\alpha} = c_{i,\beta}) : i \in \{1,2\};\: \alpha, \beta < \kappa(\s_i);\: \alpha \neq \beta\}.
\]
By \Cref{compact} and finite smoothness, there is a $\T$-interpretation $\B$ satisfying $\Gamma$: indeed, were that not true, \Cref{compact} would guarantee that some finite subset of $\Gamma$ is unsatisfiable; yet such a set would only demand the existence of finitely many new elements, which can be achieved by making use of finite smoothness. 
Since $\B \vDash \Gamma_1$, we have $|\s_1^\B| \ge \aleph_0$ and $|\s_2^\B| \ge \aleph_1$.

Since $\B$ may be too large, we construct a subinterpretation $\B^-$ with
\begin{align*}
    \s_1^{\B^-} &= \{c_{1,\alpha}^\B\}_{\alpha < \omega} \cup \{f^\B(c_{2,\alpha}^\B)\}_{\alpha < \omega_1} \\
    \s_2^{\B^-} &= \{c_{2,\alpha}^\B\}_{\alpha < \omega_1}.
\end{align*}
And using the assumption that $\T$ has built-in Skolem functions, we can prove that $\B^-$ is an elementary subinterpretation of $\B$, so $\B^- \vDash \Gamma$; we can then prove that $|\s_2^{\B^-}| = \aleph_1$, but we unfortunately cannot guarantee that $|\s_1^{\B^-}| = \aleph_0$. This is because $\B^-$ has not only the $\aleph_1$ elements $\{c_{2,\alpha}^\B\}_{\alpha < \omega_1}$ of sort $\s_2$, but also the elements $\{f^\B(c_{2,\alpha}^\B)\}_{\alpha < \omega_1}$ of sort $\s_1$. The function symbol $f$ has created a ``spillover'' of elements from $\s_2$ to $\s_1$.

To fix this, we need to ensure that $|\{f^\B(c_{2,\alpha}^\B)\}_{\alpha < \omega_1}| \le \aleph_0$. To that end, define $\Gamma$ to instead be $\{\varphi\} \cup \Gamma_1 \cup \Gamma_2$, where
\[
    \Gamma_2 = \{f(b) = f(d) : b,d \in \{c_{2,\alpha}\}_{\alpha < \omega_1}\}.
\]
Then, if there is a model $\B$ satisfying $\Gamma$, we have $|\{f^\B(c_{2,\alpha}^\B)\}_{\alpha < \omega_1}| = 1 \le \aleph_0$. To show $\Gamma$ is $\T$-satisfiable, it suffices by the compactness theorem to show that $\T \cup \Gamma'$ is satisfiable for every finite subset $\Gamma' \subseteq \Gamma$. So let $\Gamma'_1 \subseteq \Gamma_1$ and $\Gamma'_2 \subseteq \Gamma_2$ be finite subsets. We will construct a $\T$-interpretation $\B'$ such that $\B' \vDash \{\varphi\} \cup \Gamma'_1 \cup \Gamma'_2$. For concreteness, suppose that $\{c_{1,0}, c_{1,1}, \dots, c_{1,99}\}$ and $\{c_{2,0}, c_{2,1}, \dots, c_{2,9}\}$ are the new constants that appear in $\Gamma'_1 \cup \Gamma'_2$. By finite smoothness, there is a $\T$-interpretation $\B'$ satisfying $\varphi$ such that $|\s_1^{\B'}| = 100$ and $|\s_2^{\B'}| = 901$. By the pigeonhole principle, there is a subset $Y \subseteq \s_2^{\B'}$ with $|Y| \ge 10$ such that $f^{\B'}$ is constant on $Y$; if 901 pigeons are put in 100 holes, then some hole has at least 10 pigeons (although this is not true for 900 pigeons). Then, $\B'$ can interpret the constants $\{c_{1,0}, c_{1,1}, \dots, c_{1,99}\}$ as distinct elements of $\s_1^{\B'}$ and the constants $\{c_{2,0}, c_{2,1}, \dots, c_{2,9}\}$ as distinct elements of $Y$. This proves that $\Gamma$ is $\T$-satisfiable.

\makeatletter
\newcommand*\bigcdot{\mathpalette\bigcdot@{.7}}
\newcommand*\bigcdot@[2]{\mathbin{\vcenter{\hbox{\scalebox{#2}{$\m@th#1\bullet$}}}}}
\makeatother

\begin{wrapfigure}{r}{.4\textwidth}
\centering
\begin{tikzpicture}[scale=0.7, every mark/.append style={draw=white}, mark size=2.4pt]
\begin{axis}[
         xmin=-0.2, xmax=1.8,
    ymin=-0.2, ymax=1.8,
    axis lines=center,
    axis on top=true,
    domain=0:1,
    xtick={0.1, 0.5, 1.0, 1.2, 1.4},
    xticklabels={$1$,$10$,$\aleph_{0}$,$\aleph_{1}$,$\aleph_{2}$},
    extra x ticks={0.3, 0.75, 1.6},
    extra x tick style={tick style={draw=none}},
    extra x tick labels={$\cdots$, $\cdots$, $\cdots$},
    ytick={0.1, 0.5, 1.0, 1.2, 1.4},
    yticklabels={$1$,$10$,$\aleph_{0}$,$\aleph_{1}$,$\aleph_{2}$},
    extra y ticks={0.3, 0.75, 1.6},
    extra y tick style={tick style={draw=none}},
    extra y tick labels={$\vdots$, $\vdots$, $\vdots$}
    ]
    \addplot[mark=none, black, dotted] coordinates {(0.5,0) (0.5,1.7)};
    \addplot[mark=none, black, dotted] coordinates {(0,0.5) (1.7,0.5)};
    \addplot[mark=none, black, thick, dotted] coordinates {(1.0,0) (1.0,1.7)};
    \addplot[mark=none, black, thick, dotted] coordinates {(0,1.0) (1.7,1.0)};
    \addplot[mark=none, black, dotted] coordinates {(1.2,0) (1.2,1.7)};
    \addplot[mark=none, black, dotted] coordinates {(0,1.2) (1.7,1.2)};
    \addplot[mark=none, black, dotted] coordinates {(1.4,0) (1.4,1.7)};
    \addplot[mark=none, black, dotted] coordinates {(0,1.4) (1.7,1.4)};
    \node at (axis cs:-0.1, 1.75){$\s_{2}$};
    \node at (axis cs: 1.75, -0.1){$\s_{1}$};
    \node at (axis cs: 1.2, 1.6){$\color{red}\substack{\bigcdot\\\bigcdot\\\bigcdot}$};
    \node at (axis cs: 1.6, 1.3){$\color{red}\bigcdot\bigcdot\bigcdot$};
    \node at (axis cs:0.4,0.4){$\A$};
    \node at (axis cs: 1.6,1.6){$\color{red}\B$};
    \node at (axis cs:0.9,1.1){$\color{red}\B^{-}$};
    \path[->,every node/.style={font=\sffamily\small}]
    (axis cs:0.5,0.5) edge[bend left] node [right] {} (axis cs:1.1,1.6)
    (axis cs:1.6,1.15) edge[bend left] node [right] {} (axis cs:1.05,1.15);
    \addplot[only marks] 
table {
 0.5 0.5
};
\addplot[only marks, mark=*,mark options={red,scale=1.7}]
table {
 1.0 1.2
};
\addplot[only marks, mark=*,mark options={red}] 
table {
 1.0 1.4
 1.2 1.2
 1.2 1.4
 1.4 1.2
 1.4 1.4
};
\end{axis}
\end{tikzpicture}
\caption{How we move from interpretation to interpretation}\label{proof flow}
\vspace{-3mm}
\end{wrapfigure}
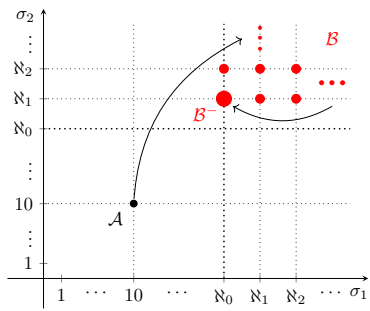

We illustrate the top level structure of the proof idea in \Cref{proof flow}, applied
to the working example.
The $x$ axis represents cardinalities of 
interpretations of $\sigma_1$,
and the $y$ axis does the same for $\sigma_2$.
Starting from the interpretation $\A$ with $|\s_{1}^{\A}|=|\s_{2}^{\A}|=10$, we construct some interpretation $\B$, represented by the array of red dots as there is some degree of uncertainty regarding the precise cardinalities of its domains, with $|\s_{1}^{\B}|\geq\aleph_{0}$ and $|\s_{2}^{\B}|\geq\aleph_{1}$. From $\B$ we hope to construct $\B^{-}$, which has $|\s_{1}^{\B^{-}}|=\aleph_{0}$ and $|\s_{2}^{\B^{-}}|=\aleph_{1}$: the latter can be achieved 
using techniques similar 
to the many-sorted L{\"o}wenheim-Skolem theorems (see \Cref{LowenheimSkolem} below), while the former requires the aforementioned pigeonhole principle arguments.

The above proof sketch illustrates the main ideas behind the proof of \Cref{thm-smoothness-from-finite-smoothness}. The generalization to more sorts and function symbols requires some extra bookkeeping. More interestingly, the generalization to functions of arity greater than one requires a version of Ramsey's theorem, which is a generalization of the pigeonhole principle.

\subsection{Ramsey's theorem and generalizations}

In this section, we state Ramsey's theorem and a generalization of it.

Ramsey's theorem is sometimes stated in terms of coloring the edges of hypergraphs, but for our purposes it is more convenient to state it as follows. In the following lemma, the notations $\nset{X}{n}$ and $[k]$ are defined as in Section~\ref{sec-notation}.

\begin{lemma}[Ramsey's theorem {\cite[Theorem~B]{ramsey1929}}] \label{lem-ramsey}
    For any $k,n,\allowbreak m \in \mathbb{N}$, there is an $R(k,n,m) \in \mathbb{N}$ such that for any set $X$ with $|X| \ge R(k,n,m)$ and function $f : \nset{X}{n} \to [k]$, there is a subset $Y \subseteq X$ with $|Y| \ge m$ such that $f$ is constant on $\nset{Y}{n}$.
\end{lemma}

Note that in Ramsey's theorem, the set $[k]$ can be replaced by any set of cardinality $k$.

We want to generalize Ramsey's theorem to functions $f : X^n \to [k]$. The most natural generalization would state that there is a large subset $Y \subseteq X$ such that $f$ is constant on $Y^n$. But this generalization is false, as the following example shows.

\begin{example}
    Let $X = \mathbb{Z}$, and let $f : X^2 \to [2]$ be given by
    \[
        f(m,n) = \begin{cases}
            1 &\quad\text{if} \ m < n \\
            2 &\quad\text{otherwise}.
        \end{cases}
    \]
    Then, $f(m,n) \neq f(n,m)$ for all $m,n \in X$ with $m \neq n$. Thus, there is no subset $Y \subseteq X$ with $|Y| \ge 2$ such that $f$ is constant on $Y^2$.
\end{example}

To avoid counterexamples like this, our generalization needs to consider the order of the arguments of $f$. This motivates the following definition.

\begin{definition}
    Let $(X, <)$ be a totally ordered set, and let $\overarrow{x} = (x_1, \dots, x_n)$ and $\overarrow{y} = (y_1, \dots, y_n)$ be elements of $X^n$. We write $\overarrow{x} \sim \overarrow{y}$ if for every $1 \le i < j \le n$ we have
    \begin{align*}
        x_i < x_j &\Longleftrightarrow y_i < y_j \quad \text{and} \\
        x_i = x_j &\Longleftrightarrow y_i = y_j.
    \end{align*}
\end{definition}

Observe that $\sim$ is an equivalence relation on $X^n$ with finitely many equivalence classes.\footnote{The number of equivalence classes is given by the ordered Bell numbers (\url{https://oeis.org/A000670}).}

Now we can state our first generalization of Ramsey's theorem.

\begin{restatable}{lemma}{dirramsey} \label{lem-ramsey-directed}
    For any $k,n,m \in \mathbb{N}$, there is an $R^*(k,n,m) \in \mathbb{N}$ such that for any totally ordered set $(X, <)$ with $|X| \ge R^*(k,n,m)$ and function $f : X^n \to [k]$, there is a subset $Y \subseteq X$ with $|Y| \ge m$ such that $f$ is constant on each $\sim$-equivalence class of $Y^n$.
\end{restatable}

Next, we further generalize Ramsey's theorem to multiple functions $f_1, \dots, f_r$.

\begin{restatable}{lemma}{multramsey} \label{lem-mult-ramsey-directed}
    For any $k,m \in \mathbb{N}$ and $\overarrow{n} = (n_1, \dots, n_r) \in \mathbb{N}^r$, there is a number $R^{**}(k,\overarrow{n},m) \in \mathbb{N}$, such that for any totally ordered set $(X, <)$ with $|X| \ge R^{**}(k,\overarrow{n},m)$ and functions $f_i : X^{n_i} \to [k]$ for $i \in [r]$, there is a subset $Y \subseteq X$ with $|Y| \ge m$, such that $f_i$ is constant on each $\sim$-equivalence class of $Y^{n_i}$ for all $i \in [r]$.
\end{restatable}

\subsection{The proof of \Cref{thm-smoothness-from-finite-smoothness}}

Fix a $\Sigma$-theory $\T$ and a set of sorts $S\subseteq\S_{\Sigma}$. Assume that $\Sigma$ is countable. Suppose that $\T$ is stably finite and finitely smooth, both with respect to $S$. Let $\varphi$ be a $\T$-satisfiable quantifier-free formula and $\A$ a $\T$-interpretation satisfying $\varphi$. Let $\kappa$ be a function from $S$ to the class of cardinals such that $\kappa(\s) \ge |\s^\A|$ for every $\s \in S$.

Write $S = \{\s_1, \s_2, \dots\}$ and, without loss of generality, assume $\kappa(\s_1) \le \kappa(\s_2) \le \cdots$. For notational convenience, we write all $\Sigma$-terms in the form $t(\overarrow{x_1}, \overarrow{x_2}, \dots)$,\footnote{Even if $S$ is infinite, the denoted term is still finite since each term only has a finite number of variables occurring in it.}
where $\overarrow{x_i}$ is a tuple of variables of sort $\s_i$. If $\kappa(\s_i) < \aleph_0$ for all $i$, then we are done by the fact $\T$ is finitely smooth.
Otherwise, let $\ell$ be the largest natural number such that $\kappa(\s_\ell) < \aleph_0$ if there is such a number, and let $\ell = 0$ otherwise.

The proof of \Cref{thm-smoothness-from-finite-smoothness} proceeds in two steps. First, we construct a set of formulas $\Gamma$ such that $\varphi \in \Gamma$ and prove that there is a $\T$-interpretation $\B$ satisfying $\Gamma$. Second, we prove that $\B$ has an elementary subinterpretation $\B^-$ such that $|\s_i^{\B^-}| = \kappa(\s_i)$ for all $i$. Since $\varphi \in \Gamma$, it will follow that $\T$ is smooth.

The assumption that $\T$ is stably finite and finitely smooth is used to construct $\T$-interpretations of the following form, which will be useful for a compactness argument.

\begin{lemma} \label{lem-big-model}
    There is a $\T$-interpretation $\B$ satisfying $\varphi$ such that $|\s_i^\B| = \kappa(\s_i)$ for all $i \le \ell$, and $|\s_i^\B|$ is arbitrarily large but finite for all $i > \ell$.
\end{lemma}
\begin{proof}
    First, apply stable finiteness to get a $\T$-interpretation $\A'$ satisfying $\varphi$ such that $|\s_i^{\A'}| \le |\s_i^\A|$ and $|\s_i^{\A'}| < \aleph_0$ for all $i$. Then, apply finite smoothness to $\A'$ with $\kappa'$ given by $\kappa'(\s_i) = \kappa(\s_i)$ for all $i \le \ell$ and $\kappa'(\s_i)$ arbitrarily large but finite for all $i > \ell$.\qed
\end{proof}

It will be convenient to work with a theory with built-in Skolem functions, so we use \Cref{builtin} to get a $\Sigma^*$-theory $\T^* \supseteq \T$, where $\Sigma^* \supseteq \Sigma$ and $\Sigma^*$ is countable.
To construct our set of formulas $\Gamma$, we introduce $\kappa(\s_i)$ new constants $\{c_{i,\alpha}\}_{\alpha < \kappa(\s_i)}$ of sort $\s_i$ for each $i$. We consider these constants to be part of an even larger signature $\Sigma' \supseteq \Sigma^*$. In what follows, we construct sentences and interpretations over $\Sigma'$. Impose an arbitrary total order on each $\{c_{i,\alpha}\}_{\alpha < \kappa(\s_i)}$ to be used for the $\sim$ relation. For the definition below, recall that given a set $X$, we define $X^{*}= \bigcup_{n\in\mathbb{N}}X^{n}$.

\begin{definition}
    We define a set of formulas $\Gamma = \{\varphi\} \cup \Gamma_1 \cup \Gamma_2 \cup \Gamma_3$, where
    \begin{align*}
        \Gamma_1 = &\{\neg (c_{i,\alpha} = c_{i,\beta}) : 1 \le i \le |S|;\: \alpha, \beta < \kappa(\s_i);\: \alpha \neq \beta\} \\
        \Gamma_2 = &\left\{t\left(\overarrow{c_1}, \dots, \overarrow{c_i}, \overarrow{b_{i+1}}, \overarrow{b_{i+2}}, \dots\right) = t\left(\overarrow{c_1}, \dots, \overarrow{c_i}, \overarrow{d_{i+1}}, \overarrow{d_{i+2}}, \dots\right) : \right. \\
        &t \ \text{is a $\Sigma^*$-term of sort} \ \s_i;\: i > \ell;\: \overarrow{c_k}, \overarrow{b_k}, \overarrow{d_k} \in (\{c_{k, \alpha}\}_{\alpha < \kappa(\s_k)})^*\\
        &\left.  \text{for all} \ k;\: \overarrow{b_j} \sim \overarrow{d_j} \ \text{for all} \ j > i \right\} \displaybreak\\
        \Gamma_3 = &\left\{\Forall{x : \s_i} \bigvee_{\alpha < \kappa(\s_i)} x = c_{i,\alpha} : i \le \ell\right\}.
    \end{align*}
\end{definition}

Note that the disjunctions in $\Gamma_3$ are finite given the condition $i \le \ell$.

\begin{restatable}{lemma}{consistent} \label{lem-consistent}
    There is a $\T^*$-interpretation $\B$ such that $\B \vDash \Gamma$.
\end{restatable}
This lemma, whose proof is in the appendix, 
forms the core of the argument. By the compactness theorem, it suffices to prove that for any finite subset $\Gamma' \subseteq \Gamma$, there is a $\T^*$-interpretation $\B'$ such that $\B' \vDash \Gamma'$. The tricky part is making $\B'$ satisfy $\Gamma' \cap \Gamma_2$. The strategy is to use \Cref{lem-big-model} to construct a model $\B'$ in which $|\s_{i+1}^{\B'}|$ is very large in terms of $|\s_i^{\B'}|$ for each $i > \ell$. \Cref{lem-mult-ramsey-directed} will ensure that there is some way of interpreting the constants $\{c_{i,\alpha}\}_{\alpha < \kappa(\s_i)}$ so that $\B' \vDash \Gamma' \cap \Gamma_2$.

We are now ready to prove \Cref{thm-smoothness-from-finite-smoothness}.

\begin{proof}[\Cref{thm-smoothness-from-finite-smoothness}]
    By \Cref{lem-consistent}, there is a $\T^*$-in\-ter\-pre\-ta\-tion $\B$ such that $\B \vDash \Gamma$. Let
    \[
        B = \left\{t^\B\left((\overarrow{c_1})^\B, (\overarrow{c_2})^\B, \dots\right) : t \ \text{is a $\Sigma^*$-term};\: \overarrow{c_i} \in (\{c_{i, \alpha}\}_{\alpha < \kappa(\s_i)})^* \ \text{for all} \ i\right\}.
    \]
    For every $f \in \F_\Sigma$, the set $B$ is closed under $f^\B$. Thus, we can define $\B^-$ to be the subinterpretation of $\B$ obtained by restricting the sorts, functions, and predicates to $B$.\footnote{In other words, $\B^-$ is the Skolem hull of $\bigcup_i \{c_{i,\alpha}^\B\}_{\alpha < \kappa(\s_i)}$ in $\B$ \cite[p. 180]{marker2002}.} Since the $\Sigma^*$-theory $\T^*$ has built-in Skolem functions, $\B^-$ is an elementary subinterpretation of $\B$ by \Cref{TV test}. We claim $|\s_i^{\B^-}| = \kappa(\s_i)$ for all $i$.

    First, $\{c_{i,\alpha}^{\B^-}\}_{\alpha < \kappa(\s_i)}$ is a set of $\kappa(\s_i)$ distinct elements in $\s_i^{\B^-}$, because $\B^- \vDash \Gamma_1$. Thus, $|\s_i^{\B^-}| \ge \kappa(\s_i)$ for all $i$.

    Second, $|\s_i^{\B^-}| \le |\{c_{i,\alpha}\}_{\alpha < \kappa(\s_i)}| = \kappa(\s_i)$ for all $i \in [\ell]$, as $\B^- \vDash \Gamma_3$.

    Finally, it remains to show that $|\s_i^{\B^-}| \le \kappa(\s_i)$ for all $i > \ell$. Inductively suppose that $|\s_j^{\B^-}| \le \kappa(\s_j)$ for all $j < i$. Now, every element of $\s_i^{\B^-}$ is of the form
    \[
        t^\B\left((\overarrow{c_1})^\B, \dots, (\overarrow{c_i})^\B, (\overarrow{c_{i+1}})^\B, (\overarrow{c_{i+2}})^\B, \dots\right),
    \]
    where $t$ is a $\Sigma^*$-term of sort $\s_i$. Since $\Sigma^*$ is countable, there are at most $\aleph_0$ choices for $t$. We have at most $\kappa(\s_i)$ choices for
    $
        (\overarrow{c_1})^\B, \dots, (\overarrow{c_i})^\B.
    $
    Finally, we have finitely many choices for 
    $(\overarrow{c_{i+1}})^\B, (\overarrow{c_{i+2}})^\B, \ldots$
    up to $\sim$-equivalence. Since $\B^- \vDash \Gamma_2$, it follows that there are at most $\kappa(\s_i)$ elements of $\s_i^{\B^-}$.
    Therefore, $\B^-$ is a $\T^*$-interpretation satisfying $\varphi$ with $|\s_i^{\B^-}| = \kappa(\s_i)$ for all $i$. Taking the reduct of $\B^-$ to $\Sigma$ gives the desired $\T$-interpretation.\qed
\end{proof}

\subsection{Applications to theory combination}

Since \Cref{thm-no-unicorn} implies that stably infinite and strongly finitely witnessable theories are strongly polite, we can restate the theorem on strongly polite theory combination with weaker hypotheses. This was already proved in \cite{BarTolZoh} via a different method, but is now obtained as an immediate corollary of \Cref{thm-no-unicorn}.

\begin{corollary} \label{cor-polite}
    Let $\Sigma_1$ and $\Sigma_2$ be disjoint countable signatures. Let $\T_1$ and $\T_2$ be $\Sigma_1$- and $\Sigma_2$-theories respectively, and let $\varphi_1$ and $\varphi_2$ be quantifier-free $\Sigma_1$- and $\Sigma_2$-formulas respectively. Suppose $\T_1$ is stably infinite and strongly finitely witnessable, both with respect to $\S_{\Sigma_1} \cap \S_{\Sigma_2}$, and let $V = \vars_{\S_{\Sigma_1} \cap \S_{\Sigma_2}}(\wit(\varphi_1))$. Then, $\varphi_1 \land \varphi_2$ is $(\T_1 \cup \T_2)$-satisfiable if and only if there is an arrangement $\delta_V$ on $V$ such that $\wit(\varphi_1) \land \delta_V$ is $\T_1$-satisfiable and $\varphi_2 \land \delta_V$ is $\T_2$-satisfiable.
\end{corollary}

We can also use our results to give a new characterization of shiny theories, which allows us to restate shiny combination theorem with weaker hypotheses.

To define shininess, we first need a few other notions. Let $\Sigma$ be a signature with $\S_\Sigma$ finite, and let $S \subseteq \S_\Sigma$. Write $S = \{\s_1, \dots, \s_n\}$. Then, the \emph{$S$-size} of a $\Sigma$-interpretation $\A$ is given by the tuple $(|\s_1^\A|, \dots, |\s_n^\A|)$. Such $n$-tuples are partially ordered by the product order: $(x_1, \dots, x_n) \preceq (y_1, \dots, y_n)$ if and only if $x_i \le y_i$ for all $i \in [n]$. Given a quantifier-free formula $\varphi$, let $\minmods{\T}{S}(\varphi)$ be the set of minimal $S$-sizes of $\T$-interpretations satisfying $\varphi$. It follows from results in \cite{wqo} that $\minmods{\T}{S}(\varphi)$ is a finite set of tuples.\footnote{\cite{casalrasga2018} proves this 
assuming that $\T$ is stably finite, using Hilbert's basis theorem. This assumption can be dropped by using the fact that if $(X, \le)$ is a well-quasi-order, then so is $(X^n, \prec)$, where $\prec$ is the product order. Here $X$ is the class of cardinals.}

Then, we say a $\Sigma$-theory $\T$ is \emph{shiny} with respect to some subset of sorts $S \subseteq \S_\Sigma$ if $\S_\Sigma$ is finite, $\T$ is stably finite and smooth, both with respect to $S$, and $\minmods{\T}{S}$ is computable. \Cref{thm-smoothness-from-finite-smoothness} implies that we can replace smoothness by finite smoothness, which may make it easier to prove that some theories are shiny. We can therefore improve the shiny theory combination theorem from \cite[Theorem~2]{casalrasga2018} as an immediate corollary of \Cref{thm-smoothness-from-finite-smoothness}.

\begin{corollary} \label{cor-shiny}
    Let $\Sigma_1$ and $\Sigma_2$ be disjoint countable signatures, where $\S_{\Sigma_1}$ and $\S_{\Sigma_2}$ are finite. Let $\T_1$ and $\T_2$ be $\Sigma_1$- and $\Sigma_2$-theories respectively, and assume the satisfiability problems for quantifier-free formulas of both $\T_{1}$ and $\T_{2}$ are decidable. Suppose $\T_1$ is stably finite and finitely smooth, both with respect to $\S_{\Sigma_1} \cap \S_{\Sigma_2}$, and $\minmods{\T_1}{\S_{\Sigma_1} \cap \S_{\Sigma_2}}$ is computable. Then, the satisfiability problem for quantifier-free formulas of $\T_{1}\cup\T_{2}$ is decidable.
\end{corollary}

\section{Many-sorted L{\"o}wenheim--Skolem theorems}\label{LowenheimSkolem}

In this section, we state many-sorted generalizations of the L{\"o}wen\-heim--Skolem theorem. Our first results, in \Cref{sec:down}, hold with no assumptions on the signature. Later, in \Cref{sec:split}, we state stronger results for restricted signatures, which we then 
use for a many-sorted variant of the 
{\L}o\'s--Vaught test in
\Cref{sec:los-vaught}. But first, 
in \Cref{Somedifficulties},
we explain 
the limitations of relying solely on translations to single-sorted first-order logic.

\subsection{Lost in translation}\label{Somedifficulties}

% Given a many-sorted signature $\Sigma$, let $\Sigmadag$ be the single-sorted signature with a function symbol $f^\dag$ of arity $n$ for each function symbol $f$ of arity $(\s_1, \dots, \s_n, \s)$ in $\Sigma$, a predicate symbol $P^\dag$ of arity $n$ for each predicate symbol $P$ of arity $(\s_1, \dots, \s_n)$ in $\Sigma$, and a unary predicate symbol $P_{\s}$ for each sort $\s$ in $\Sigma$. Given a $\Sigmadag$-structure $\sA$, it is then possible to understand $P_{\s}^{\sA}$ to be the domain of sort $\s$. 
% Of course, some restrictions are necessary, such as the axiom schemata $\Exists{x}P_{\s}(x)$ and $\Forall{x}(P_{\s}(x)\rightarrow\neg P_{\t}(x))$ for sorts $\s\neq\t$. 
We may transform a many-sorted signature into a single-sorted signature by adding unary predicates signifying the sorts; of course, some restrictions are necessary, 
distinctness of sorts, etc. 
This procedure~\cite{Wang,Enderton,Monk} is often used to lift results from single-sorted to many-sorted logic. 
As one example, standard versions of the downward L{\"o}wenheim--Skolem theorem for many-sorted logic, found in \cite{Monzano93}, are proven using this translation; we can, however, strengthen these results while still using only translations:

\begin{restatable}[Downward]{theorem}{translateddownwardmanysorted}\label{Lowenheim}
    Let $\Sigma$ be a many-sorted signature with $|\S_{\Sigma}|<\aleph_{0}$. Suppose we have a $\Sigma$-structure $\sA$ with $\max\{|\s^{\sA}| : \s\in \S_{\Sigma}\}\geq \aleph_{0}$, a cardinal $\kappa$ satisfying $\max\{|\Sigma|, \aleph_{0}\}\leq\kappa\leq\min\{|\s^{\sA}| : \s\in \Sasa{\geq}\}$,
    and sets $A_{\s}\subseteq \s^{\sA}$ with $|A_{\s}|\leq\kappa$ for each $\sigma\in\S_{\Sigma}$. Then, there is an elementary substructure $\sB$ of $\sA$ such that $\s^{\sB}=\s^{\sA}$ for every $\s\in \Sasa{<}$, $\aleph_{0}\leq|\s^{\sB}|\leq \kappa$ for all $\s\in\Sasa{\geq}$, $|\s^{\sB}|=\kappa$ for some $\s\in\S_{\Sigma}$, and $A_{\s}\subseteq \s^{\sB}$ for all $\s\in\S_{\Sigma}$.
\end{restatable}

\begin{restatable}[Upward]{theorem}{translatedULS}\label{translated ULS}
    Let $\Sigma$ be a many-sorted signature with $|\S_{\Sigma}|<\aleph_{0}$. Suppose we have a $\Sigma$-structure $\sA$ with $\max\{|\s^{\sA}| : \s\in \S_{\Sigma}\}\geq \aleph_{0}$ and a cardinal 
    $\kappa \geq \max\{|\Sigma|, \max\{|\s^{\sA}| : \s\in\S_{\Sigma}\}\}$. Then, there is a $\Sigma$-structure $\sB$ containing $\sA$ as an elementary substructure such that $\s^{\sB}=\s^{\sA}$ for all $\s\in\Sasa{<}$, $\aleph_{0}\leq|\s^{\sB}|\leq\kappa$ for all $\s\in\Sasa{\geq}$, and $|\s^{\sB}|=\kappa$ for some sort $\s\in\S_{\Sigma}$.
\end{restatable}

As convenient as translation arguments are, 
% they cannot be used to prove all properties of many-sorted logic, as Feferman noticed when generalizing Craig's interpolation theorem \cite{Craig1,Craig2} to many-sorted logic \cite{Feferman}. Indeed, 
the above L{\"o}wenheim--Skolem theorems seem unsatisfactory, as they only allow us to choose a single cardinal, rather than one for each sort.%, which seems to be an inherent limitation of this kind of argument.

\subsection{Downward, upward, and combined versions}
\label{sec:down}\label{sec:up}\label{sec:med}

The following are generalizations of the downward and upward L{\"o}wenheim--Skolem theorems to many-sorted logic, which are proved by adapting the proofs of the single-sorted case. Notice that we set all infinite domains to the same cardinality, while finite domains preserve their cardinalities.

\begin{restatable}[Downward]{theorem}{generalizedLST}\label{generalized LST}
Fix a first-order many-sorted signature $\Sigma$. Suppose we have a $\Sigma$-structure $\sA$, a cardinal $\kappa$ such that
$\max\{\aleph_{0}, |\Sigma|\}\leq \kappa\leq\min\{|\s^{\sA}| : \s\in \Sasa{\geq}\}$, and sets $A_{\s} \subseteq \s^\sA$ with $|A_\s| \le \kappa$ for each $\s\in \Sasa{\geq}$. Then, there is an elementary substructure $\sB$ of $\sA$ that satisfies $|\s^{\sB}|=\kappa$ and $\s^{\sB}\supseteq A_{\s}$ for every $\s\in \Sasa{\geq}$, and also $\s^{\sB}=\s^{\sA}$ for every $\s\in \Sasa{<}$.
\end{restatable}

\begin{restatable}[Upward]{theorem}{GeneralizationULST}\label{Generalization ULST}\label{LowenheimSkolemUpwards}
Fix a first-order many-sorted signature $\Sigma$. Given a $\Sigma$-structure $\sA$, pick a cardinal $\kappa\geq \max\{|\Sigma|, \aleph_{0},\allowbreak \sup\{ |\s^{\sA}| : \s\in \Sasa{\geq}\}\}$.
Then, there is a $\Sigma$-structure $\sB$ containing $\sA$ as an elementary substructure that satisfies $|\s^{\sB}|=\kappa$ for every $\s\in \Sasa{\geq}$, and also $\s^{\sB} = \s^{\sA}$ for every $\s\in \Sasa{<}$.
\end{restatable}

\Cref{generalized LST,Generalization ULST} can be combined to yield yet another variant of the L{\"o}wenheim--Skolem theorem, which may be called the combined version.

\begin{restatable}[Combined]{corollary}{middleLST}\label{middleLST}
    Fix a many-sorted signature $\Sigma$. Given a $\Sigma$-structure $\sA$, pick a cardinal $\kappa\geq\max\{|\Sigma|, \aleph_{0}\}$. Then, there is a $\Sigma$-structure $\sB$ elementarily equivalent to $\sA$ with $|\s^{\sB}|=\kappa$ for every $\s\in \Sasa{\geq}$, and $\s^{\sB} = \s^{\sA}$ for $\s\in \Sasa{<}$.
\end{restatable}

\begin{wrapfigure}[12]{r}{.5\textwidth}
%\begin{figure}[t]
%\vspace{-2em}
\centering
\begin{tikzpicture}[scale=.85, every mark/.append style={draw=white}, mark size=2.4pt]
\begin{axis}[
    xmin=-0.1, xmax=1.3,
    ymin=0, ymax=1,
    axis lines=center,
    axis on top=true,
    axis y line=none,
    domain=0:1,
    clip=false,
    xtick={0, 0.1, 0.4, 0.6, 0.9, 1.2},
    xticklabels={$\s_{1}$, $\s_{2}$, $\s_{n}$, $\s^{\prime}_{1}$, $\s^{\prime}_{i}$,$\s^{\prime}_{m}$},
    extra x ticks={0.25, 0.75, 1.05},
    extra x tick style={tick style={draw=none}},
    extra x tick labels={$\cdots$, $\cdots$, $\cdots$}
    ]
    \addplot[mark=none, black, thick] coordinates {(-0.1,0.25) (1.3,0.25)};
        \addplot[mark=none, black, thick] coordinates {(-0.1,0.435) (1.3,0.435)};
    \addplot[mark=none, black, thick, dotted] coordinates {(0.0,0) (0.0,0.2)};
    \addplot[mark=none, black, thick, dotted] coordinates {(0.4,0) (0.4,0.125)};
    \addplot[mark=none, black, thick, dotted] coordinates {(0.1,0) (0.1,0.075)};
    \addplot[mark=none, black, thick, dotted] coordinates {(0.6,0) (0.6,0.435)};
    \addplot[mark=none, black, thick, dotted] coordinates {(0.9,0) (0.9,0.435)};
    \addplot[mark=none, black, thick, dotted] coordinates {(1.2,0) (1.2,0.435)};
    \addplot[mark=none, red, thick, dotted] coordinates {(1.2,0.435) (1.2,0.475)};
    \node at (axis cs:-0.1, 0.285){$\aleph_{0}$};
    \node at (axis cs:-0.1, 0.4){$\kappa$};
    \addplot[only marks] 
table {
 0.6 0.435
 0.9 0.435
 1.2 0.435
};
\addplot[only marks, rotated halfcircle=-90]
table {
 0.0 0.2
  0.1 0.075
 0.4 0.125
};
\addplot[only marks, mark=*,mark options={red}] 
table {
 0.6 0.375
 0.9 0.4
 1.2 0.475
};
\end{axis}
\end{tikzpicture}
\caption{Illustration of \Cref{middleLST}.}
\label{fig:illmid}
\end{wrapfigure}
%\end{figure}

We illustrate \Cref{middleLST} in \Cref{fig:illmid}.
In black, we represent the cardinalities of the resulting structure, and in red, those of the original one. 
When they coincide, we use marks split between the two colors. 
This representation shows a set of sorts in the horizontal axis, and the heights of the marks represent the cardinalities of the respective domains. 
We clearly separate cardinals larger and smaller than $\aleph_{0}$ with a rule. 
Assume, without loss of generality, that initially $\s_1\dots\s_n$ have finite 
cardinalities and $\s^{\prime}_{1}$ has the least and $\s^{\prime}_{m}$ the greatest 
infinite cardinality.\footnote{For greater clarity, the diagram only depicts the cases where there are finitely many sorts and the signature is countable.} 
\Cref{middleLST} allows us to pick
an infinite cardinal $\kappa$ in between
the least and greatest infinite cardinalities,
and set all infinite cardinlaities in the
interpretation to $\kappa$.

The above theorems require that the desired cardinalities of the infinite sorts are all equal. The following example shows that this limitation is necessary.

\begin{example}
\label{Failure of LST}
Take the signature $\Sigma$ with sorts $S=\{\s_{1}, \s_{2}\}$, no predicates, and only one function $f$ of arity $(\s_{1},\s_{2})$. Take the $\Sigma$-structure $\sA$ with:  $\s_{1}^{\sA}$ and $\s_{2}^{\sA}$ of cardinality $\aleph_{1}$, and $f^{\sA}$ a bijection. It is then true that $\sA \vDash \varphi_{\inj} \land \varphi_{\sur}$, where $\varphi_{\inj}=\Forall{x : \s_1}\Forall{y : \s_1}\big[[f(x)=f(y)]\rightarrow[x=y]\big]$ and $\varphi_{\sur}=\Forall{u : \s_2}\Exists{x : \s_1}[f(x)=u]$, codifying that $f$ is injective and surjective respectively. Notice then that, although $\max\{|\Sigma|, \aleph_{0}\}=\aleph_{0}$, there cannot be an elementary substructure $\sB$ of $\sA$ with $|\s_{1}^{\sB}|=\aleph_{0}$ and $|\s_{2}^{\sB}|=\aleph_{1}$: for if $\sB\vDash\varphi_{\inj}\wedge\varphi_{\sur}$, $f^{\sB}$ must be a bijection between $\s_{1}^{\sB}$ and $\s_{2}^{\sB}$. A similar argument shows that the corresponding
generalization of the upwards theorem fails as well.
\end{example}

\subsection{A stronger result for split signatures}
\label{sec:split}

\Cref{Failure of LST} relies on ``mixing sorts'' by using a function symbol with arities spanning different sorts. We can state stronger versions of the many-sorted L{\"o}wenheim--Skolem theorems when such mixing of sorts is restricted.

\begin{definition}
 A signature $\Sigma$ is said to be {\em split by $\Lambda$ into a family of signatures $\{\Sigma_{\lambda} : \lambda\in\Lambda\}$} if $\Lambda$ is a partition of $\S_{\Sigma}$, $\S_{\Sigma_{\lambda}}=\lambda$ for each $\lambda\in\Lambda$, $\F_{\Sigma}=\bigcup_{\lambda\in\Lambda}\F_{\Sigma_{\lambda}}$, and $\P_{\Sigma}=\bigcup_{\lambda\in\Lambda}\P_{\Sigma_{\lambda}}$. If $\Sigma$ is split by $\Lambda$ and each $\lambda\in\Lambda$ is a singleton, then we say that $\Sigma$ is {\em completely split by $\Lambda$}.
\end{definition}

If $\Sigma$ is split by $\Lambda$, then the function/predicate symbols of $\Sigma_{\lambda}$ must be disjoint from $\Sigma_{\lambda'}$ for $\lambda \neq \lambda'$. 
Given a partition $\Lambda$ of $\S_{\Sigma}$ and $\lambda\in\Lambda$, let $\Sasa{\geq}(\lambda)=\Sasa{\geq}\cap\lambda$. 
We state the downward, upward, and combined 
%many-sorted 
%L{\"o}wenheim--Skolem 
theorems for split signatures.

\begin{restatable}[Downward]{theorem}{downwardsplitLS}\label{downwardssplit}
Fix a first-order many-sorted signature $\Sigma$ split by $\Lambda$. Suppose we have a $\Sigma$-structure $\sA$, a cardinal $\kappa_\lambda$ such that $\max\{\aleph_{0}, |\Sigma_{\lambda}|\}\leq \kappa_{\lambda}\leq\min\{|\s^{\sA}| : \s\in \Sasa{\geq}(\lambda)\}$ for each $\lambda \in \Lambda$, and sets $A_{\s} \subseteq \s^\sA$ with $|A_\s| \le \kappa_\lambda$ for each $\s\in \Sasa{\geq}(\lambda)$. Then, there is an elementary substructure $\sB$ of $\sA$ that satisfies $|\s^{\sB}|=\kappa_\lambda$ and $\s^{\sB}\supseteq A_{\s}$ for $\s\in \Sasa{\geq}(\lambda)$, and $\s^{\sB}=\s^{\sA}$ for $\s\in \Sasa{<}$.
\end{restatable}

\begin{restatable}[Upward]{theorem}{upwardsplitLS}\label{upwardssplit}
Suppose $\Sigma$ is split by $\Lambda$. Given a $\Sigma$-structure $\sA$, pick a cardinal 
$\kappa_{\lambda}\geq \max\{|\Sigma_{\lambda}|, \aleph_{0}, \sup\{|\s^{\sA}| : \s\in \Sasa{\geq}(\lambda)\}\}$ for each $\lambda \in \Lambda$. Then, there is a $\Sigma$-structure $\sB$ containing $\sA$ as an elementary substructure that satisfies $|\s^{\sB}|=\kappa_\lambda$ for $\s\in \Sasa{\geq}(\lambda)$, and $\s^{\sB} = \s^{\sA}$ for $\s\in \Sasa{<}$.
\end{restatable}

\begin{restatable}[Combined]{corollary}{middlesplitLS}\label{middlesplit}
    Suppose $\Sigma$ is split by $\Lambda$. Given a $\Sigma$-structure $\sA$, pick a cardinal $\kappa_{\lambda}\geq\max\{|\Sigma_{\lambda}|, \aleph_{0}\}$ for each $\lambda \in \Lambda$. Then, there is a $\Sigma$-structure $\sB$ elementarily equivalent to $\sA$ with $|\s^{\sB}|=\kappa_\lambda$ for every $\s\in \Sasa{\geq}(\lambda)$, and also $\s^{\sB} = \s^{\sA}$ for every $\s\in \Sasa{<}$.
\end{restatable}

\begin{wrapfigure}[10]{r}{.5\textwidth}
%\begin{figure}[t]
\vspace{-3em}
\centering
\begin{tikzpicture}[scale=0.85, every mark/.append style={draw=white}, mark size=2.4pt]
\begin{axis}[
    xmin=-0.1, xmax=1.8,
    ymin=0, ymax=1,
    axis lines=center,
    axis on top=true,
    axis y line=none,
    domain=0:1,
    clip=false,
    xtick={0, 0.1, 0.4, 0.6, 0.7, 1.0, 1.3, 1.6, 1.7},
    xticklabels={$\s_{1}$, $\s_{2}$, $\s_{n}$, $\s^{\prime}_{1}$, $\s^{\prime\prime}_{1}$,$\s^{\prime}_{i}$,$\s^{\prime\prime}_{j}$,$\s^{\prime}_{m}$,$\s^{\prime\prime}_{m}$},
    extra x ticks={0.25, 0.85, 1.15, 1.45},
    extra x tick style={tick style={draw=none}},
    extra x tick labels={$\cdots$, $\cdots$, $\cdots$, $\cdots$}
    ]
    \addplot[mark=none, black, thick] coordinates {(-0.1,0.25) (1.8,0.25)};
    \addplot[mark=none, black, thick] coordinates {(-0.1,0.44) (1.8,0.44)};
    \addplot[mark=none, black, thick] coordinates {(-0.1,0.505) (1.8,0.505)};
    \node at (axis cs:-0.1, 0.285){$\aleph_{0}$};
    \node at (axis cs:-0.1, 0.405){$\kappa^{\prime}$};
    \node at (axis cs:-0.1, 0.54){$\kappa^{\prime\prime}$};
    \addplot[mark=none, black, thick, dotted] coordinates {(0.0,0) (0.0,0.2)};
    \addplot[mark=none, black, thick, dotted] coordinates {(0.1,0) (0.1,0.075)};
    \addplot[mark=none, black, thick, dotted] coordinates {(1.0,0) (1.0,0.44)};
    \addplot[mark=none, black, thick, dotted] coordinates {(1.6,0) (1.6,0.44)};
    \addplot[mark=none, red, thick, dotted] coordinates {(1.6,0.44) (1.6,0.475)};
    \addplot[mark=none, black, thick, dotted] coordinates {(0.4,0) (0.4,0.125)};
    \addplot[mark=none, black, thick, dotted] coordinates {(0.6,0) (0.6,0.44)};
    \addplot[mark=none, black, thick, dotted] coordinates {(0.7,0) (0.7,0.505)};
    \addplot[mark=none, black, thick, dotted] coordinates {(1.3,0) (1.3,0.505)};
    \addplot[mark=none, black, thick, dotted] coordinates {(1.7,0) (1.7,0.505)};
    \addplot[mark=none, red, thick, dotted] coordinates {(1.7,0.505) (1.7,0.545)};
    \addplot[only marks] 
table {
 0.6 0.44
 0.7 0.505
 1.0 0.44
 1.3 0.505
 1.6 0.44
 1.7 0.505
};
\addplot[only marks, rotated halfcircle=-90]
table {
 0.0 0.2
  0.1 0.075
 0.4 0.125
};
\addplot[only marks, mark=*,mark options={red}] 
table {
 0.6 0.375
 0.7 0.445
 1.0 0.4
 1.3 0.465
 1.6 0.475
 1.7 0.545
};
\end{axis}
\end{tikzpicture}
\caption{Illustration of \Cref{middlesplit}.}
\label{fig:illmidsplit}
\end{wrapfigure}
%\end{figure}

% The results from \Cref{sec:down} follow immediately from these results, since any signature $\Sigma$ can be split into the trivial partition $\Sigma_\lambda = \Sigma$.
 
\Cref{middlesplit} is illustrated in
\Cref{fig:illmidsplit}.
We add sorts $S^{\prime\prime}=\{\s^{\prime\prime}_{1},\ldots,\s^{\prime\prime}_{m}\}$, and assume 
our signature is split into $\Sigma_{\lambda_{1}}$ and $\Sigma_{\lambda_{2}}$, where $\Sasa{\geq}(\lambda_{1})=\{\s^{\prime}_{1},\ldots,\s^{\prime}_{m}\}$ and $\Sasa{\geq}(\lambda_{2})=S^{\prime\prime}$ 
%(we choose $|\lambda_{1}|=|\lambda_{2}|$ to keep things simple, but in general, this may not hold). 
(the sorts with finite cardinalities can belong to either).
Then, $\kappa^{\prime}$ is the cardinal associated with $\Sigma_{\lambda_{1}}$, and $\kappa^{\prime\prime}$ with $\Sigma_{\lambda_{2}}$.
Thus, we are able
to choose a cardinality for each class of sorts.

\subsection{An application: the {\L}o\'s--Vaught test}\label{application}\label{sec:los-vaught}

We describe an application of our L{\"o}wenheim--Skolem theorems 
for theory-completeness: the {\L}o\'s--Vaught test.
%The question of whether a theory is complete is 
This is
particularly relevant to SMT, as if a complete theory $\T$ has a decidable set of axioms, then it is decidable whether $\vdash_{\T} \varphi$ ~\cite[Lemma~2.2.8]{marker2002}. 
%While proving completeness can be difficult in general, 
% The {\L}o\'s--Vaught test makes it easy 
% to prove completeness for theories satisfying a model-theoretic criterion.
%
The single-sorted {\L}o\'s--Vaught is the following.

%To state the {\L}o\'s--Vaught test, we need the following definition.
\begin{definition}
    Let $\Sigma$ be a signature and $\kappa$ a function from $\S_{\Sigma}$ to the class of cardinals. A $\Sigma$-theory $\T$ is \emph{$\kappa$-categorical} if it has exactly one model $\sA$ (up to isomorphism) with the property that $|\s^{\sA}| = \kappa(\s)$ for every $\s\in \S_{\Sigma}$. If there is only one sort $\s \in \S_{\Sigma}$, we abuse notation by using $\kappa$ to denote the cardinal $\kappa(\s)$.
\end{definition}

\begin{theorem}[\cite{los1954,vaught1954}] \label{losvaughtunsorted}
    Suppose $\Sigma$ is single-sorted and $\T$ is a $\Sigma$-theory with only infinite models. If $\T$ is $\kappa$-categorical for some $\kappa \ge |\Sigma|$, then $\T$ is complete.
\end{theorem}

%Despite being a simple corollary of the L{\"o}wenheim--Skolem theorem, 
The {\L}o\'s--Vaught test is quite useful, e.g., for the completeness 
of dense linear orders without endpoints and algebraically closed fields.
%of a given characteristic. 
We generalize it to many sorts. Translating to one-sorted logic 
%as in \Cref{Somedifficulties} 
and using \Cref{losvaughtunsorted} 
gives us:

\begin{restatable}{corollary}{losvaughttranslation} \label{losvaughttranslation}
Let $\Sigma$ be a signature with $|\S_{\Sigma}|<\aleph_{0}$.
    Suppose $\T$ is a $\Sigma$-theory, all of whose models $\sA$ satisfy $\max\{|\s^{\sA}| : \s \in \S_{\Sigma}\} \ge \aleph_{0}$.  Suppose further that for some cardinal $\kappa \ge |\Sigma|$, $\T$ has exactly one model $\sA$ (up to isomorphism) such that $\max\{|\s^{\sA}| : \s \in \S_{\Sigma}\} = \kappa$. Then, $\T$ is complete.
\end{restatable}

This is not the result one would hope for, because it excludes some many-sorted $\kappa$-categorical theories, as the following example demonstrates.

\begin{example} \label{acf0}
    Suppose $\Sigma$ has $S = \{\s_1, \s_2\}$, no predicate symbols, and function symbols 0, 1, $+$, and $\times$, of the expected arities. Let
    $
        \T = \mathsf{ACF_0} \cup \big\{\psi^{\s_2}_{\geq n} : n \in \mathbb{N}\big\},
    $
    where $\mathsf{ACF_0}$ is the theory of algebraically closed fields of characteristic zero (with respect to $\s_1$) and
    $
        \psi^{\s}_{\geq n}=\Exists{x_{1} : \s}\cdots \Exists{x_n : \s} \bigwedge_{1\leq i<j\leq n}\neg(x_{i}=x_{j})
    $,
    which asserts that there are at least $n$ elements of sort $\s$. 
    $\T$ is $\kappa$-categorical, where $\kappa(\s_1) = \aleph_1$ and $\kappa(\s_2) = \aleph_{0}$. But $\T$ is also $\kappa'$-categorical, where $\kappa'(\s_1) = \kappa'(\s_2) = \aleph_1$. Thus, $\T$ has multiple models $\sA$ satisfying $\max\{|\s^{\sA}| : \s \in \S_{\Sigma}\} = \aleph_1$. 
    Similar reasoning holds for other infinite cardinals, 
    so \Cref{losvaughttranslation} does not apply.
\end{example}

For completely split signatures, we prove a more natural {\L}o\'s--Vaught test:

\begin{definition}
    A $\Sigma$-structure $\sA$ is \emph{strongly infinite} if $|\s^{\sA}| \geq \aleph_{0}$ for all $\s\in \S_{\Sigma}$.
\end{definition}

\begin{restatable}{theorem}{losvaught} \label{losvaught}
    Suppose $\Sigma$ is completely split into $\{\Sigma_\s : \s \in \S_{\Sigma}\}$, $\T$ is a $\Sigma$-theory all of whose models are strongly infinite, and $\T$ is $\kappa$-categorical for some function $\kappa$ such that $\kappa(\s) \ge |\Sigma_\s|$ for every $\s \in \S_{\Sigma}$. Then, $\T$ is complete.
\end{restatable}

The assumption that $\Sigma$ is completely split is necessary for \Cref{losvaught}:

\begin{example}
    Let $\Sigma$ have sorts $\s_{1}, \s_{2}$,  
    and  function symbol $f$ of arity $(\s_{1},\s_{2})$. 
    Let
     $
        \T = \big\{\psi^{\s_1}_{\geq n} : n \in \mathbb{N}\big\} \cup \big\{\psi^{\s_2}_{\geq n} : n \in \mathbb{N}\big\}\cup \big\{\varphi_{\inj} \lor \Forall{x : \s_1}\Forall{y : \s_1} [f(x)=f(y)]\big\}
     $.
    In $\T$, $\s_1,\s_2$ are infinite, and $f$ is injective or 
    %collapses every element of $\s_1$ into the same element of $\s_2$
    constant.  
    $\T$ is $\kappa$-categorical for $\kappa(\s_1) = \aleph_1,\kappa(\s_2) = \aleph_{0}$,
    %(with a model $\sA$ given by $\s_1^\sA = \omega_1 \times \{0\}$, $\s_2^\sA = \omega \times \{1\}$,\footnote{The purpose of the Cartesian products is to make $\s_1^\sA$ and $\s_2^\sA$ disjoint.} and $f^\sA : (\alpha,0) \mapsto (0,1)$, where $\alpha < \omega_1$). 
    % But $\T$ also has a model $\sB$ given by $\s_1^\sB = \omega \times \{0\}$, $\s_2^\sB = \omega \times \{1\}$, and $f^\sB : (\alpha,0) \mapsto (\alpha,1)$, where $\alpha < \omega$. 
    but not complete, due to
    the sentence
    $\forall x ,y : \s_1 .f(x)=f(y)$.
    This does not contradict \Cref{losvaught}, as $\Sigma$ is not completely split.
\end{example}

\section{Conclusion}\label{conclusion}
We closed the problem of the existence of unicorn theories and discussed applications  to SMT. 
This included a result similar to the L{\"o}wenheim--Skolem theorem,
which inspired us to investigate the adaptation of this theorem to many-sorted logic. 
% both in general and for restricted classes of signatures. 
We also obtained a many-sorted version of the {\L}o\'s--Vaught test.

In future work, we plan to investigate  whether \Cref{thm-smoothness-from-finite-smoothness} can be extended to uncountable signatures. More broadly, we intend to continue studying the relationships among many-sorted model-theoretic properties related to SMT.

\begin{credits}
\subsubsection{\ackname} This work was supported in part by the Stanford Center for Automated Reasoning,
NSF-BSF grant numbers 2110397 (NSF) and 2020704 (BSF),
ISF grant 619/21, and
the Colman-Soref fellowship.
The first author thanks the organizers of the CURIS research program.

%\subsubsection{\discintname}
%The authors have no competing interests to declare that are relevant to the content of this article.
\end{credits}

%\pagebreak

% \begin{credits}
% \subsubsection{\ackname} The first author thanks the organizers of the CURIS research program.

% \subsubsection{\discintname}
% The authors have no competing interests to declare that are
% relevant to the content of this article.
% \end{credits}
\bibliographystyle{splncs04}
\bibliography{bib}

\begin{thebibliography}{10}
\providecommand{\url}[1]{\texttt{#1}}
\providecommand{\urlprefix}{URL }
\providecommand{\doi}[1]{https://doi.org/#1}

\bibitem{cvc5}
Barbosa, H., Barrett, C.W., Brain, M., Kremer, G., Lachnitt, H., Mann, M., Mohamed, A., Mohamed, M., Niemetz, A., N{\"{o}}tzli, A., Ozdemir, A., Preiner, M., Reynolds, A., Sheng, Y., Tinelli, C., Zohar, Y.: cvc5: {A} versatile and industrial-strength {SMT} solver. In: {TACAS} {(1)}. Lecture Notes in Computer Science, vol. 13243, pp. 415--442. Springer, Munich (2022)

\bibitem{BarFT-RR-17}
Barrett, C., Fontaine, P., Tinelli, C.: {The SMT-LIB Standard: Version 2.6}. Tech. rep., Department of Computer Science, The University of Iowa (2017), available at \url{http://smt-lib.org}

\bibitem{BT18}
Barrett, C., Tinelli, C.: Satisfiability modulo theories. In: Clarke, E.M., Henzinger, T.A., Veith, H., Bloem, R. (eds.) Handbook of Model Checking, pp. 305--343. Springer, New York (2018). \doi{10.1007/978-3-319-10575-8_11}, \url{http://theory.stanford.edu/~barrett/pubs/BT18.pdf}

\bibitem{quineconj}
Barrett, T.W., Halvorson, H.: Quine's conjecture on many-sorted logic. Synthese  \textbf{194}(9),  3563--3582 (2017)

\bibitem{casalrasga2018}
Casal, F., Rasga, J.a.: Many-sorted equivalence of shiny and strongly polite theories. J. Automat. Reason.  \textbf{60}(2),  221--236 (2018)

\bibitem{ehrenfeuchtmostowski}
Ehrenfeucht, A., Mostowski, A.: Models of axiomatic theories admitting automorphisms. Fund. Math.  \textbf{43},  50--68 (1956)

\bibitem{Enderton}
Enderton, H.B.: A Mathematical Introduction to Logic. Academic Press, New York (1972)

\bibitem{gentle}
Fontaine, P.: Combinations of theories for decidable fragments of first-order logic. In: Ghilardi, S., Sebastiani, R. (eds.) Frontiers of Combining Systems. pp. 263--278. Springer Berlin Heidelberg, Berlin, Heidelberg (2009)

\bibitem{stronglypolite}
Jovanovi{\'c}, D., Barrett, C.: Polite theories revisited. Tech. Rep. TR2010-922, Depatrment of Computer Science, New York University (Jan 2010), \url{http://www.cs.stanford.edu/~barrett/pubs/JB10-TR.pdf}

\bibitem{flexible}
Krsti{\'{c}}, S., Goel, A., Grundy, J., Tinelli, C.: Combined satisfiability modulo parametric theories. In: Grumberg, O., Huth, M. (eds.) Tools and Algorithms for the Construction and Analysis of Systems. pp. 602--617. Springer Berlin Heidelberg, Berlin, Heidelberg (2007). \doi{https://doi.org/10.1007/978-3-540-71209-1_47}

\bibitem{wqo}
Kruskal, J.B.: The theory of well-quasi-ordering: {A} frequently discovered concept. J. Combinatorial Theory Ser. A  \textbf{13},  297--305 (1972). \doi{10.1016/0097-3165(72)90063-5}, \url{https://doi.org/10.1016/0097-3165(72)90063-5}

\bibitem{los1954}
\L~o\'{s}, J.: On the categoricity in power of elementary deductive systems and some related problems. Colloquium Mathematicum  \textbf{3},  58--62 (1954)

\bibitem{marker2002}
Marker, D.: Model theory: {An} introduction, Graduate Texts in Mathematics, vol.~217. Springer-Verlag, New York (2002)

\bibitem{Monk}
Monk, J.D.: Mathematical Logic. Springer, New York (1976). \doi{10.1007/978-1-4684-9452-5}, \url{https://doi.org/10.1007/978-1-4684-9452-5}

\bibitem{Monzano93}
Monzano, M.: Introduction to many-sorted logic. In: Meinke, K., Tucker, J.V. (eds.) Many-sorted Logic and its Applications. Wiley professional computing, Wiley, New York (1993)

\bibitem{NelsonOppen}
Nelson, G., Oppen, D.C.: Simplification by cooperating decision procedures. ACM Trans. Program. Lang. Syst.  \textbf{1}(2),  245–257 (oct 1979). \doi{10.1145/357073.357079}, \url{https://doi.org/10.1145/357073.357079}

\bibitem{ramsey1929}
Ramsey, F.P.: On a problem of formal logic. Proc. London Math. Soc. (2)  \textbf{30}(4),  264--286 (1929)

\bibitem{polite}
Ranise, S., Ringeissen, C., Zarba, C.G.: {Combining data structures with nonstably infinite theories using many-sorted logic}. In: Gramlich, B. (ed.) {5th International Workshop on Frontiers of Combining Systems - FroCoS'05}. Lecture Notes in Artificial Intelligence, vol.~3717, pp. 48--64. {Springer}, Vienna/Austria (Sep 2005). \doi{10.1007/11559306}, \url{https://hal.inria.fr/inria-00000570}

\bibitem{algebraicdatatypes}
Sheng, Y., Zohar, Y., Ringeissen, C., Lange, J., Fontaine, P., Barrett, C.: Polite combination of algebraic datatypes. Journal of Automated Reasoning  \textbf{66}(3),  331--355 (Aug 2022). \doi{10.1007/s10817-022-09625-3}, \url{https://doi.org/10.1007/s10817-022-09625-3}

\bibitem{shiny}
Tinelli, C., Zarba, C.G.: Combining decision procedures for sorted theories. Tech. rep., Berlin, Heidelberg (2004). \doi{https://doi.org/10.1007/978-3-540-30227-8_53}

\bibitem{BarTolZoh}
de~Toledo, G.V., Zohar, Y., Barrett, C.W.: Combining combination properties: An analysis of stable infiniteness, convexity, and politeness. In: {CADE}. Lecture Notes in Computer Science, vol. 14132, pp. 522--541. Springer, Rome (2023)

\bibitem{BTZarxiv}
de~Toledo, G.V., Zohar, Y., Barrett, C.W.: Combining combination properties: An analysis of stable infiniteness, convexity, and politeness. CoRR  \textbf{abs/2305.02384} (2023)

\bibitem{BarTolZoh2}
de~Toledo, G.V., Zohar, Y., Barrett, C.W.: Combining finite combination properties: Finite models and busy beavers. In: FroCoS. Lecture Notes in Computer Science, vol. 14279, pp. 159--175. Springer, Prague (2023)

\bibitem{vaught1954}
Vaught, R.L.: Applications of the {L}\"{o}wenheim-{S}kolem-{T}arski theorem to problems of completeness and decidability. Nederl. Akad. Wetensch. Proc.  \textbf{57},  467--472 (1954)

\bibitem{Wang}
Wang, H.: Logic of many-sorted theories. The Journal of Symbolic Logic  \textbf{17}(2),  105--116 (1952), \url{http://www.jstor.org/stable/2266241}

\end{thebibliography}

\newgeometry{
left=20mm,
right=20mm,
bottom=22mm,
top=15mm
}
\appendix

%\crefname{theorem}{THEOREM}{THEOREMS}
%\crefname{corollary}{COROLLARY}{COROLLARIES}
%\crefname{example}{EXAMPLE}{EXAMPLES}
%\crefname{lemma}{LEMMA}{LEMMAS}
%\Crefname{theorem}{Theorem}{Theorems}
%\Crefname{corollary}{Corollary}{Corollaries}
%\Crefname{example}{Example}{Examples}
%\Crefname{lemma}{Lemma}{Lemmas}

\section*{Appendix}

\section{Proof of \Cref{lem-73}}
This lemma was used in \cite{BarTolZoh},
and is explicitly found, with a proof, in its extended
technical report \cite{BTZarxiv},
as Lemma~73.
For completeness, we include its proof
in this appendix.

\lemseventhree*

\begin{proof}
    Let $\T$ be stably infinite and strongly finitely witnessable, both with respect to $S$. Let $\varphi$ be a $\T$-satisfiable quantifier-free formula, $\A$ a $\T$-interpretation satisfying $\varphi$, and $\kappa$ a function from $\Saa{<}\cap S$ to the class of cardinals such that $|\s^{\A}| \leq \kappa(\s) < \aleph_0$ for every $\s \in \Saa{<}\cap S$. We have $\A \vDash \Exists{\overarrow{w}} \wit(\varphi)$, where $\overarrow{w}=\vars(\wit(\varphi))\setminus\vars(\varphi)$. Hence, by modifying the interpretation of the variables in $\overarrow{w}$, we obtain a $\T$-interpretation $\A'$ satisfying $\wit(\varphi)$. Let $V = \vars(\wit(\varphi))$, and let $\delta_V$ be the arrangement on $V$ induced by the equalities in $\A'$. Then, $\wit(\varphi) \land \delta_V$ is $\T$-satisfiable, so there exists a $\T$-interpretation $\A''$ satisfying $\wit(\varphi) \land \delta_V$ such that $\s^{\A''}=\vars_{\s}(\wit(\varphi) \land \delta_V)^{\A''}$ for every $\s \in S$. The map from $\s^{\A''}$ to $\s^{\A'}$ given by $x^{\A''} \mapsto x^{\A'}$, where $x \in \vars_\s(\wit(\varphi))$, is well-defined and injective, so we have $|\s^{\A''}| \le |\s^{\A'}| = |\s^{\A}|$ for every $\s \in S$. In particular, $\kappa(\s) - |\s^{\A''}| \ge \kappa(\s) - |\s^{\A}| \ge 0$. For each $\s \in \Saa{<}\cap S$, introduce $\kappa(\s) - |\s^{\A''}|$ fresh variables $W_\s$ of sort $\s$. Let $W = \bigcup_{\s \in \Saa{<}\cap S} W_\s$, and extend the arrangement $\delta_V$ to an arrangement $\delta_{V \cup W}$ by asserting that all variables in $W$ are distinct from each other and other variables in $V$. Since $\T$ is stably infinite, $\wit(\varphi) \land \delta_{V \cup W}$ is $\T$-satisfiable, so there exists a $\T$-interpretation $\B$ satisfying $\wit(\varphi) \land \delta_{V \cup W}$ such that $\s^{\B}=\vars_{\s}(\wit(\varphi) \land \delta_{V \cup W})^{\B}$ for every $\s \in S$. Since $\B$ satisfies $\wit(\varphi)$, it also satisfies $\varphi$. We also have $|\s^{\B}| = |\s^{\A''}| + |W_\s| = \kappa(\s)$ for every $\s \in \Saa{<}\cap S$. Therefore, $\T$ is finitely smooth with respect to $S$.\qed
\end{proof}

\section{Proof of \Cref{lem-ramsey-directed}}
\dirramsey*
\begin{proof}
    Let
    \[
        R^*(k,n,m) = R\left(k^{n^n},n,m+n-1\right).
    \]
    \sloppy For any function $f : X^n \to [k]$, let $f_\rho(x_1, \dots, x_n) = f(x_{\rho(1)}, \dots, x_{\rho(n)})$, where $\rho : [n] \to [n]$ is an arbitrary function. Fix an ordering $\rho_1, \dots, \rho_{n^n}$ on the set of functions from $[n]$ to itself. Then, let $F : \nset{X}{n} \to [k]^{n^n}$ be given by, for $x_1 < \dots < x_n$,
    \[
        F(\{x_1, \dots, x_n\}) = (f_{\rho_1}(x_1, \dots, x_n), \dots, f_{\rho_{n^n}}(x_1, \dots, x_n)).
    \]
    By \Cref{lem-ramsey}, for any totally ordered set $(X, <)$ with $|X| \ge R^*(k,n,m)$, there is a subset $Y' \subseteq X$ with $|Y'| \ge m+n-1$ such that $F$ is constant on $\nset{Y'}{n}$. Let $Y \subseteq Y'$ be the initial $m$ elements of $Y'$ according to the order on $Y'$ inherited from $X$. Let $\overarrow{x}, \overarrow{y} \in Y^n$ with $\overarrow{x} \sim \overarrow{y}$, and let the distinct elements of $\overarrow{x}$ be $x_1 < \dots < x_\ell$ and let those of $\overarrow{y}$ be $y_1 < \dots < y_\ell$ for some $\ell \in [n]$. Add additional elements from $Y'$ to get $\{x_1, \dots, x_n\} \supseteq \{x_1, \dots, x_\ell\}$ and $\{y_1, \dots, y_n\} \supseteq \{y_1, \dots, y_\ell\}$, where $x_1 < \dots < x_n$ and $y_1 < \dots < y_n$. Then, $F(\{x_1, \dots, x_n\}) = F(\{y_1, \dots, y_n\})$, since $\{x_1, \dots, x_n\}, \{y_1, \dots, y_n\} \in \nset{Y'}{n}$. Hence,
    \[
        f_{\rho_i}(x_1, \dots, x_n) = f_{\rho_i}(y_1, \dots, y_n)
    \]
    for all $i \in [n^n]$. In particular, let $\rho_i$ be such that
    \begin{align*}
        \left(x_{\rho_i(1)}, \dots, x_{\rho_i(n)}\right) &= \overarrow{x} \quad\text{and} \\
        \left(y_{\rho_i(1)}, \dots, y_{\rho_i(n)}\right) &= \overarrow{y}.
    \end{align*}
    Then,
    \[
        f(\overarrow{x}) = f_{\rho_i}(x_1, \dots, x_n) = f_{\rho_i}(y_1, \dots, y_n) = f(\overarrow{y}),
    \]
    as desired.\qed
\end{proof}

\section{Proof of \Cref{lem-mult-ramsey-directed}}
\multramsey*
\begin{proof}
    Let $n = n_1 + \dots + n_r$, and let $R^{**}(k,\overarrow{n},m) = R^*(k^r, n, m+1)$. Given functions $f_i : X^{n_i} \to [k]$, let $F : X^n \to [k]^r$ be given by
    \[
        F(\overarrow{x_1}, \dots, \overarrow{x_r}) = (f_1(\overarrow{x_1}), \dots, f_r(\overarrow{x_r})).
    \]
    As proven in \Cref{lem-ramsey-directed}, for any totally ordered set $(X, <)$ with $|X| \ge R^{**}(k,\overarrow{n},m)$, there is a subset $Y' \subseteq X$ with $|Y'| \ge m+1$ such that $F$ is constant on $\sim$-equivalence classes of ${Y'}^n$. Let $y' \in Y'$ be the maximum element according to the order on $Y'$ inherited from $X$, and let $Y = Y' \setminus \{y'\}$. Given some $i \in [r]$, let $\overarrow{x}, \overarrow{y} \in Y^{n_i}$ with $\overarrow{x} \sim \overarrow{y}$. Then,
    \begin{align*}
        &((y')^{\oplus n_1}, \dots, (y')^{\oplus n_{i-1}}, \overarrow{x}, (y')^{\oplus n_{i+1}}, \dots, (y')^{\oplus n_r}) \sim \\
        &((y')^{\oplus n_1}, \dots, (y')^{\oplus n_{i-1}}, \overarrow{y}, (y')^{\oplus n_{i+1}}, \dots, (y')^{\oplus n_r}),
    \end{align*}
    since $y'$ is strictly greater than every element in $\overarrow{x}$ and $\overarrow{y}$. Therefore,
    \begin{align*}
        &F((y')^{\oplus n_1}, \dots, (y')^{\oplus n_{i-1}}, \overarrow{x}, (y')^{\oplus n_{i+1}}, \dots, (y')^{\oplus n_r}) = \\
        &F((y')^{\oplus n_1}, \dots, (y')^{\oplus n_{i-1}}, \overarrow{y}, (y')^{\oplus n_{i+1}}, \dots, (y')^{\oplus n_r}),
    \end{align*}
    so $f_i(\overarrow{x}) = f_i(\overarrow{y})$, as desired.\qed
\end{proof}

\section{Proof of \Cref{lem-consistent}}
\consistent*
\begin{proof}
    By the compactness theorem, it suffices to prove that $\T^* \cup \Gamma'$ is satisfiable for every finite subset $\Gamma' \subseteq \Gamma$. So let $\Gamma'_1 \subseteq \Gamma_1$, $\Gamma'_2 \subseteq \Gamma_2$, and $\Gamma'_3 \subseteq \Gamma_3$ be finite subsets. We will construct a $\T^*$-interpretation $\B'$ such that $\B' \vDash \{\varphi\} \cup \Gamma'_1 \cup \Gamma'_2 \cup \Gamma'_3$. The tricky part is making $\B'$ satisfy $\Gamma'_2$. The strategy is to use \Cref{lem-big-model} to construct a model $\B'$ in which $|\s_{i+1}^{\B'}|$ is very large in terms of $|\s_i^{\B'}|$ for each $i > \ell$. \Cref{lem-mult-ramsey-directed} will ensure that there is some way of interpreting the constants $\{c_{i,\alpha}\}_{\alpha < \kappa(\s_i)}$ so that $\B' \vDash \Gamma'_2$.
    
    For each $i$, let $C_i = \{c_{i,\alpha} : \alpha < \kappa(\s_i);\: c_{i,\alpha} \ \text{appears in} \ \Gamma'_1 \cup \Gamma'_2 \cup \Gamma'_3\}$. Since $|\bigcup_i C_i| < \aleph_0$, there is a maximum natural number $i$ such that $C_i \neq \emptyset$, which we denote $s$. For each $i > \ell$, let
    \[
        T_i = \{t : t \ \text{is a $\Sigma^*$-term of sort $\s_k$ appearing in} \ \Gamma'_2 \ \text{for some} \ k < i\}.
    \]
    Since $T_i \subseteq T_{i+1}$ for each $i$, and each $T_{i}$ is finite, we can enumerate the terms of $\bigcup_i T_i$ so that for each $i$, there is an $r_i$ such that $t_1, \dots, t_{r_i}$ is an enumeration of $T_i$.
    
    Let $\B'$ be a $\T$-interpretation satisfying $\varphi$ obtained according to \Cref{lem-big-model}, where $|\s_i^{\B'}|$ when $i > \ell$ is specified as follows. Suppose inductively that $|\s_k^{\B'}|$ has been determined for all $k < i$. Given a term of the form $t(\overarrow{x_1}, \dots, \overarrow{x_s})$,
    let
    \[
        m_{t,i} = \left|\s_1^{\B'}\right|^{\left|\overarrow{x_1}\right|} \times \dots \times \left|\s_{i-1}^{\B'}\right|^{\left|\overarrow{x_{i-1}}\right|} \times |C_{i+1}|^{\left|\overarrow{x_{i+1}}\right|} \times \dots \times |C_s|^{\left|\overarrow{x_s}\right|},
    \]
    and let $n_{t,i} = |\overarrow{x_i}|$. Let $\overarrow{n_i} = ((n_{t_1, i})^{\oplus m_{t_1,i}}, \dots, (n_{t_{r_i}, i})^{\oplus m_{t_{r_i},i}})$.\footnote{Recall that $(x)^{\oplus n}$ denotes the tuple consisting of $x$ repeated $n$ times.} Then, choose $|\s_i^{\B'}|$ so that
    \[
        \left|\s_i^{\B'}\right| \ge R^{**}\left(\left|\s_{\ell+1}^{\B'}\right| + \dots + \left|\s_{i-1}^{\B'}\right|, \overarrow{n_i}, |C_i|\right),
    \]
    where $R^{**}$ is the function from \Cref{lem-mult-ramsey-directed}.

    Now, we specify how $\B'$ interprets the constants $C_i$. Note that it does not matter how $\B'$ interprets the constants in $\{c_{i,\alpha}\}_{\alpha < \kappa(\s_i)} \setminus C_i$, since these constants do not appear in $\{\varphi\} \cup \Gamma'_1 \cup \Gamma'_2 \cup \Gamma'_3$. Impose an arbitrary total order on each $\s_i^{\B'}$ to be used for the $\sim$ relation. If $i \le \ell$, then interpret the elements of $C_i$ as distinct elements of $\s_i^{\B'}$, which is possible because $|C_i| \le |\{c_{i,\alpha}\}_{\alpha < \kappa(\s_i)}| = \kappa(\s_i) = |\s_i^{\B'}|$. Otherwise, if $i > \ell$, we specify the interpretation of the constants $C_i$ by induction on $s-i$. That is, we specify the interpretation of $C_s$, then that of $C_{s-1}$, and so on. Suppose that the interpretation of the constants $C_j$ has been determined for all $j > i$. Given a term $t \in T_i$, define the following family of functions in $(\s_i^{\B'})^{n_{t,i}} \to \s_{\ell+1}^{\B'} \cup \dots \cup \s_{i-1}^{\B'}$:
    \begin{align*}
        \mathfrak{f}_{t,i} = &\left\{\overarrow{a} \mapsto t^{\B'}\left(\overarrow{a_1}, \dots, \overarrow{a_{i-1}}, \overarrow{a}, (\overarrow{c_{i+1}})^{\B'}, \dots, (\overarrow{c_s})^{\B'}\right) : \right. \\
        &\left. \overarrow{a_k} \in \s_k^{\B'} \ \text{for all} \ k < i;\: \overarrow{c_j} \in (C_j)^* \ \text{for all} \ j > i\right\}.
    \end{align*}
    Observe that $|\mathfrak{f}_{t,i}| = m_{t,i}$. By our choice of $|\s_i^{\B'}|$, we can apply \Cref{lem-mult-ramsey-directed} to the functions $\mathfrak{f}_i \coloneqq \mathfrak{f}_{t_1,i} \cup \dots \cup \mathfrak{f}_{t_{r_i},i}$ to conclude that there is a subset $Y_i \subseteq \s_i^{\B'}$ with $|Y_i| \ge |C_i|$ such that each $f \in \mathfrak{f}_i$ is constant on $\sim$-equivalence classes of $Y_i^n$, where $n$ is the arity of $f$. Then, interpret the constants $C_i$ as distinct elements of $Y_i$ in a way that is compatible with their respective total orders (i.e., $c_{i, \alpha} < c_{i, \beta}$ if and only if $c_{i, \alpha}^{\B'} < c_{i, \beta}^{\B'}$).

    This completes the description of $\B'$. It remains to show that $\B' \vDash \Gamma'_1 \cup \Gamma'_2 \cup \Gamma'_3$.
    
    First, we show that $\B' \vDash \Gamma'_1$. By construction, $\B'$ interprets the constants $C_i$ as distinct elements of $\s_i^{\B'}$ for all $i \in [s]$. Therefore, $\B' \vDash \Gamma'_1$.

    Second, we show that $\B' \vDash \Gamma'_2$. For each $i > \ell$, let
    \begin{align*}
        {\Gamma'_2}^i = &\left\{t\left(\overarrow{c_1}, \dots, \overarrow{c_{i-1}}, \overarrow{b_i}, \overarrow{c_{i+1}}, \dots, \overarrow{c_s}\right) = t\left(\overarrow{c_1}, \dots, \overarrow{c_{i-1}}, \overarrow{d_i}, \overarrow{c_{i+1}}, \dots, \overarrow{c_s}\right) : \right. \\
        &\left. t \in T_i;\: \overarrow{c_k}, \overarrow{b_k}, \overarrow{d_k} \in (C_k)^* \ \text{for all} \ k;\: \overarrow{b_i} \sim \overarrow{d_i}\right\}.
    \end{align*}
    Since $C_i^{\B'} \subseteq Y_i$ and each $f \in \mathfrak{f}_i$ is constant on $\sim$-equivalence classes of $Y^n$, where $n$ is the arity of $f$, we have
    \begin{align*}
        &t^{\B'}\left((\overarrow{c_1})^{\B'}, \dots, (\overarrow{c_{i-1}})^{\B'}, \left(\overarrow{b_i}\right)^{\B'}, (\overarrow{c_{i+1}})^{\B'}, \dots, (\overarrow{c_s})^{\B'}\right) = \\
        &t^{\B'}\left((\overarrow{c_1})^{\B'}, \dots, (\overarrow{c_{i-1}})^{\B'}, \left(\overarrow{d_i}\right)^{\B'}, (\overarrow{c_{i+1}})^{\B'}, \dots, (\overarrow{c_s})^{\B'}\right)
    \end{align*}
    for each $t \in T_i$ whenever $\overarrow{b_i} \sim \overarrow{d_i}$. Thus, $\B' \vDash {\Gamma'_2}^i$ for all $i > \ell$. Now, observe that $\bigcup_{i>\ell} {\Gamma'_2}^i$ entails $\Gamma'_2$. Therefore, $\B' \vDash \Gamma'_2$.

    Finally, we show that $\B' \vDash \Gamma'_3$. Suppose $\Gamma'_3$ contains a sentence of the form
    \[
        \Forall{x \in \s_i} \bigvee_{\alpha < \kappa(\s_i)} x = c_{i,\alpha},
    \]
    where $i \le \ell$. Then, $C_i = \{c_{i, \alpha}\}_{\alpha < \kappa(\s_i)}$, so $|C_i| = \kappa(\s_i) = |\s_i^{\B'}|$. Since $\B'$ interprets the constants $C_i$ distinctly, $C_i^{\B'} = \s_i^{\B'}$. Thus, the sentence above is equivalent to the fact that every element of $\s_i^{\B'}$ is denoted by some constant in $C_i$. Therefore, $\B' \vDash \Gamma'_3$.\qed
\end{proof}

\section{Proof of \Cref{Lowenheim}}\label{a proof of}

\translateddownwardmanysorted*

\begin{proof}
Consider the $\Sigmadag$-structure $\sAdag$ with: domain $\bigcup_{\s\in\S_{\Sigma}}\s^{\sA}$; for every function $f$ of arity $(\s_1, \dots, \s_n, \s)$ (in $\Sigma$), $f^{\sAdag}$ equals $f^{\sA}$ when restricted to $\s_{1}^{\sA}\times\cdots\times\s_{n}^{\sA}$, and is arbitrary otherwise; for every predicate $P$ of arity $(\s_1, \dots, \s_n)$ (in $\Sigma$), $P^{\sAdag}$ equals $P^{\sA}$; and $a\in P_{\s}^{\sAdag}$ iff $a\in\s^{\sA}$. Notice that, because $\sA$ is a $\Sigma$-structure, we get $\sAdag$ satisfies the following additional formulas:
\begin{enumerate}
    \item[I] for every $\s\in\S_{\Sigma}$, $\Exists{x}P_{\s}(x)$;
    \item[II] for any two distinct $\s, \t\in\S_{\Sigma}$, $\Forall{x}(P_{\s}(x)\rightarrow\neg P_{\t}(x))$;
    \item[III] if $\S_{\Sigma}=\{\s_{1},\ldots,\s_{n}\}$, $\Forall{x}P_{\s_{1}}(x)\vee\cdots\vee P_{\s_{n}}(x)$;
    \item[IV] for $f$ of arity $(\s_1, \dots, \s_n, \s)$, $\Forall{x_{1},\cdots, x_{n}}\bigwedge_{i=1}^{n}P_{\s_{i}}(x_{i})\rightarrow P_{\s}(f(x_{1}, \ldots , x_{n}))$ (with some obvious care being necessary if $n=0$).
\end{enumerate}
    Now, $|\Sigma|=|\Sigmadag|$ since $\F_{\Sigmadag}$ must be in bijection with $\F_{\Sigma}$, and $\P_{\Sigmadag}$ with $\S_{\Sigma}\cup\P_{\Sigma}$; and because $\sA$ has an infinite domain and $\kappa\leq\min\{|\s^{\sA}|:\s\in\S_{\Sigma}\}$, $\sAdag$ is infinite and has cardinality greater than $\kappa$. Taking $A=\bigcup_{\s\in\S_{\Sigma}}A_{\s}$, $|A|\leq \kappa$, and we can therefore apply the classical downward L{\"o}wenheim--Skolem to obtain an elementary substructure $\sBdag$ of $\sAdag$ with domain of cardinality $\kappa$, and containing $A$.

    Finally, we define a $\Sigma$-structure $\sB$ by making: $\s^{\sB}=P_{\s}^{\sBdag}$ for every sort $\s$ (these are nonempty and disjoint, given $\sAdag$ satisfies the sets of formulas in $I$ and $II$); for a function $f$ of arity $(\s_1, \dots, \s_n, \s)$, $f^{\sB}$ equals $f^{\sBdag}$ restricted to $\s_{1}^{\sB}\times\cdots\times\s_{n}^{\sB}$ (which is well-defined because $\sAdag$ satisfies the set of formulas in $IV$); and, for a predicate $P$ of arity $(\s_1, \dots, \s_n)$, $P^{\sB}$ equals the intersection of $P^{\sBdag}$ and $\s_{1}^{\sB}\times\cdots\times\s_{n}^{\sB}$, making of $\sB$ a substructure of $\sA$. It it easy to prove that $\sB$ is elementary equivalent to $\sA$, and therefore $\s^{\sB}$ has the same cardinality as $\s^{\sA}$ if the latter is finite (and thus $\s^{\sB}=\s^{\sA}$), and is infinite if the latter is infinite.

    We also get that, since $A$ is contained in $\sBdag$, $A_{\s}$ is contained in $\s^{\sB}$. Finally, $\bigcup_{\s\in\S_{\Sigma}}\s^{\sB}$ equals, given $\sBdag$ satisfies the formula in $III$, the domain of $\sBdag$, meaning $\sum_{\s\in\S_{\Sigma}}|\s^{\sB}|=\kappa$; given $\S_{\Sigma}$ is finite, this means some domain of $\sB$ has cardinality $\kappa$, finishing the proof.\qed
\end{proof}

\section{Proof of \Cref{translated ULS}}

\translatedULS*

\begin{proof}
    Construct the $\Sigmadag$-structure $\sAdag$ as in the proof of \Cref{Lowenheim}, and since the cardinality of the domain of $\sAdag$ is $\max\{|\s^{\sA}|:\s\in\S_{\Sigma}\}$ we can apply the classical upward L{\"o}wenheim--Skolem to obtain a $\Sigmadag$-structure $\sBdag$ with an elementary substructure isomorphic to $\sAdag$, and a domain of cardinality $\kappa$. Furthermore, since $\S_{\Sigma}$ is finite (say it equals $\{\s_{1}, \ldots , \s_{n}\}$), then $\sAdag$ satisfies 
    $\Forall{x}(P_{\s_{1}}(x)\vee\cdots\vee P_{\s_{n}}(x))$,
    and so must $\sBdag$. Translating $\sBdag$ back to a $\Sigma$-structure $\sB$, again as done in the proof of \Cref{Lowenheim}, we obtain that $\sB$ has an elementary substructure isomorphic to $\sA$ (and so $|\s^{\sB}|=|\s^{\sA}|$ for every $\s\in\S_{\Sigma}$ such that $\s^{\sA}$ is finite); and for some $\s\in\S_{\Sigma}$, $\sB$ has cardinality $\kappa$, because every element of $\sBdag$ must be in some domain of $\sB$.\qed
\end{proof}

\section{Proofs of \Cref{generalized LST,Generalization ULST,middleLST}}

These results are obtained as corollaries from
\Cref{downwardssplit,upwardssplit,middlesplit}, whose proofs can be found in the following sections.
The reason being that every signature can be trivially split to a singleton partition.

\section{Proof of \Cref{downwardssplit}}

\begin{definition}
Suppose $\Sigma$ is split by $\Lambda$. A $\Sigma$-formula is said to be a \emph{generalized $\Lambda$-cube} if it is a conjunction $\bigwedge_{i=1}^{n}\varphi_{i}$, where each $\varphi_{i}$ is a $\Sigma_{\lambda_{i}}$-formula, where $\lambda_i\in\Lambda$ and for $i\not=j$, we have $\lambda_i\not=\lambda_j$; similarly, a $\Sigma$-formula is said to be  a \emph{generalized $\Lambda$-clause} if it is a disjunction $\bigvee_{i=1}^{m}\varphi_{i}$, where each $\varphi_{i}$ is a $\Sigma_{\lambda_{i}}$-formula, where $\lambda_i\in\Lambda$ and for $i\not=j$, we have $\lambda_i\not=\lambda_j$.

A $\Sigma$-formula that is a disjunction of generalized $\Lambda$-cubes (respectively, a conjunction of generalized $\Lambda$-clauses) is said to be in generalized disjunctive $\Lambda$-normal form, or $\Lambda$-GDNF (respectively, generalized conjunctive normal $\Lambda$-form, or $\Lambda$-GCNF).\footnote{Whenever $\Lambda$ is clear from context, we will omit it from the nomenclature.
%so a formula in GDNF is a disjunction of generalized cubes, for example.
}
\end{definition}

% We now show that in split signatures, there is a
% normal form for each formula.

We start with some technical lemmas. 

\begin{restatable}{lemma}{GDNFandDCNF}\label{GDNF=DCNF}
    Suppose $\Sigma$ is split into $\{\Sigma_{\lambda}:\lambda\in\Lambda\}$; then a formula that is equivalent to a formula in GDNF is also equivalent to a formula in GCNF, and vice-versa.
\end{restatable}

\begin{proof}
    We prove that if $\varphi$ that is equivalent to a formula in GDNF is also equivalent to a formula in GCNF: the reciprocal has an analogous proof.

    So, suppose that $\varphi$ is equivalent to $\psi=\bigvee_{i=1}^{m}\bigwedge_{j=1}^{n_{i}}\varphi^{i}_{j}$, and define the number of generalized literals in $\psi$ as $n=\sum_{i=1}^{m}n_{i}$: notice that any quantifiers in $\psi$ must be inside one of the $\varphi_{j}^{i}$. We proceed by induction on $n$: if $n=1$, $\psi$ is already in GCNF as well, so there is nothing to prove. Suppose then that the result holds for some $n\geq 1$, and take a generalized cube $\bigwedge_{i=1}^{n_{i}}\varphi^{i}_{j}$ with $n_{i}>1$ (if there are none, again $\psi$ is already in GCNF): without loss of generality, assume that $i=m$, and that $m>1$ (otherwise $\psi$ is again in GCNF, and there is nothing to be done). Then we have that, denoting by $\theta\equiv\theta^{\prime}$ the fact that $\theta$ and $\theta^{\prime}$ are equivalent,
    \[\psi=\bigvee_{i=1}^{m}\bigwedge_{j=1}^{n_{i}}\varphi^{i}_{j}=[\bigvee_{i=1}^{m-1}\bigwedge_{j=1}^{n_{i}}\varphi^{i}_{j}]\vee[(\bigwedge_{j=1}^{n_{m}-1}\varphi^{m}_{j})\wedge\varphi^{m}_{n_{m}}]\equiv[(\bigvee_{i=1}^{m-1}\bigwedge_{j=1}^{n_{i}}\varphi^{i}_{j})\vee(\bigwedge_{j=1}^{n_{m}-1}\varphi^{m}_{j})]\wedge[(\bigvee_{i=1}^{m-1}\bigwedge_{j=1}^{n_{i}}\varphi^{i}_{j})\vee\varphi^{m}_{n_{m}}],\]
    by using the distributivity of disjunction over conjunction. Now, in the second line, the formulas on both sides of the conjunction are in GDNF and have a number of generalized literals strictly less than that of $\psi$, so they are equivalent by induction hypothesis to formulas $\psi_{1}$ and $\psi_{2}$ in GCNF. To summarize, $\psi$ is then equivalent to $\psi_{1}\wedge\psi_{2}$, which is itself in GCNF, and so $\varphi$ is equivalent to a formula in GCNF.\qed
\end{proof}

\begin{restatable}{lemma}{GDNFexists}\label{GDNF}
    If $\Sigma$ is a split signature, each of its formulas is equivalent to a formula in GDNF.
\end{restatable}

\begin{proof}
    It is well known that any first-order $\Sigma$-formula $\varphi$ is equivalent to a formula in prenex normal form (PNF), that is, to a formula 
    \[\Q{1}{x_{1}}\cdots\Q{n}{x_{n}}\:\varphi,\]
    where $Q_{i}\in\{\forall, \exists\}$ and $\varphi$ is quantifier free; without loss of generality, let us assume that all $\Sigma$-formulas are in PNF, and we write the proof by induction on $n$.

    If $n=0$, $\varphi$ is itself quantifier-free: writing $\varphi$ in disjunctive normal form (DNF), and using the commutativity of conjunction to place literals of the same signature $\Sigma_{\lambda}$ together (notice every literal on $\Sigma$ is a literal of one of the $\Sigma_{\lambda}$ because $\Sigma$ is split), we obtain $\varphi$ is equivalent to a formula in GDNF.

    Now, assume the result holds for $n\geq 1$, and then it is true that
    \[
        \varphi=\Q{1}{x_{1}}\cdots\Q{n+1}{x_{n+1}}\:\varphi=\Q{1}{x_{1}}\:\big(\Q{2}{x_{2}}\cdots\Q{n+1}{x_{n+1}}\:\varphi\big)=\Q{1}{x_{1}}\:\bigvee_{i=1}^{p}\bigwedge_{j=1}^{q}\varphi_{j}^{i}
    \]
    by induction hypothesis, where $\varphi_{j}^{i}$ are $\Sigma_{\lambda_{j}}$-formulas.\footnote{Notice that if a generalized cube of a formula in GDNF does not include formulas of exactly the same signatures as the other generalized cubes, we can always add tautologies to make the treatment of that formula more uniform.} Now, we have two cases to consider.
    \begin{enumerate}
        \item If $Q_{1}=\exists$ and $x_{1}$ is of sort, without loss of generality, in $\Sigma_{\lambda_{q}}$, we have that 
        \[\varphi=\Exists{x_{1}}\bigvee_{i=1}^{p}\bigwedge_{j=1}^{q}\varphi_{j}^{i}=\bigvee_{i=1}^{p}\Exists{x_{1}}\bigwedge_{j=1}^{q}\varphi_{j}^{i}=\bigvee_{i=1}^{p}\Exists{x_{1}}\varphi_{q}^{i}\wedge\bigwedge_{j=1}^{q-1}\varphi_{j}^{i},\]
        since $\varphi_{j}^{i}$ for $1\leq j\leq q-1$ cannot have the variable $x_{1}$. Of course, we are then done.

        \item Now, suppose $Q_{1}=\forall$. Because of \Cref{GDNF=DCNF} and our induction hypothesis, we know that we can rewrite the formula $\Q{2}{x_{2}}\cdots\Q{n+1}{x_{n+1}}\:\varphi$ as $\bigwedge_{i=1}^{P}\bigvee_{j=1}^{Q}\psi_{j}^{i}$ for some $\Sigma_{\lambda_{j}}$-formulas $\psi_{j}^{i}$. Then, assuming again without loss of generality that $x_{1}$ is of sort in $\Sigma_{\lambda_{Q}}$,
        \[\varphi=\Forall{x_{1}}\bigwedge_{i=1}^{P}\bigvee_{j=1}^{Q}\psi_{j}^{i}=\bigwedge_{i=1}^{P}\Forall{x_{1}}\bigvee_{j=1}^{Q}\psi_{j}^{i}=\bigwedge_{i=1}^{P}\Forall{x_{1}}\varphi_{Q}^{i}\vee\bigvee_{j=1}^{Q-1}\varphi_{j}^{i},\]
        which is in GCNF. Once again applying \Cref{GDNF=DCNF}, we obtain $\varphi$ may be written in GDNF, as we wanted to prove. \qed
    \end{enumerate}
\end{proof}

\downwardsplitLS*

\begin{proof}
    Given a formula $\varphi$ and free variable $x \in \vars(\varphi)$, let $f^{x}_{\varphi}$ be a Skolem function, meaning that $(\sA, \nu)\vDash\Exists{x}\varphi$ implies $(\sA, \mu)\vDash\varphi$, where $\mu$ differs from $\nu$ at most on $x$, $\vars(\varphi)=\{x, y_{1}, \ldots , y_{n}\}$, and $\mu(x)=f^{x}_{\varphi}(\nu(y_{1}), \ldots , \nu(y_{n}))$. Skolem functions can be proven to exist as in single-sorted logic. For each $\s\in \Sasa{\geq}(\lambda)$, we take a set $A^0_\s$ such that $A_{\s}\subseteq A_{\s}^{0}\subseteq \s^{\sA}$ and $|A^0_\s| = \kappa_\lambda$, which is possible given that the $A_{\s}$ in the statement of the theorem have cardinality at most $\kappa_{\lambda}$; if $\s\in \Sasa{<}(\lambda)$, we make $A_{\s}^{0}=\s^{\sA}$. We then define, for every $m\in\mathbb{N}$, if $\s$ is a sort of $\Sigma_{\lambda}$,
    \begin{multline*}
A^{m+1}_{\s}=A^{m}_{\s}\cup\{f_{\varphi}^{x}(a_{1},\ldots, a_{n}) :\text{$\varphi$ is a $\Sigma_{\lambda}$-formula, $\vars(\varphi)=$}\\\{x, y_{1},\ldots, y_{n}\}, \text{$x$ is of sort $\s$, $y_{i}$ is of sort $\s_{i}$, and $a_{i}\in A^{m}_{\s_{i}}$}\}.
\end{multline*}

We define a $\Sigma$-structure $\sB$, where $\s^{\sB}=\bigcup_{n\in\mathbb{N}}A_{\s}^{m}$; $f^{\sB}$, with $f$ of arity $(\s_1, \dots, \s_n, \s)$, equals $f^{\sA}$ restricted to $\s_{1}^{\sB}\times\cdots\times\s_{n}^{\sB}$; and $P^{\sB}$, with $P$ of arity $(\s_1, \dots, \s_n)$, equals $P^{\sA}\cap(\s_{1}^{\sB}\times\cdots\times\s_{n}^{\sB})$. We claim that if $\s \in \Sigma_{\lambda}$, then $|\s^{\sB}|=\kappa_{\lambda}$. Since $\kappa_\lambda \geq \aleph_0$, it suffices to show that $|A^{m}_{\s}|=\kappa_{\lambda}$ for each $m\in\mathbb{N}$. This is true for $m=0$ by hypothesis. The cardinality of the set $\For{\Sigma_{\lambda}}$ of formulas on the signature $\Sigma_{\lambda}$ is at most $\max\{|\Sigma_{\lambda}|, \aleph_{0}\}\leq \kappa_{\lambda}$. Thus,
\[
\kappa_\lambda = |A_{\s}^{m}| \leq |A^{m+1}_{\s}| \leq |A_{\s}^{m}|+\sum_{n\in\mathbb{N}}|\For{\Sigma_{\lambda}}|\times|A_{\s}^{m}|^{n} \leq \kappa_{\lambda},
\]
so $|A^{m+1}_{\s}| = \kappa_\lambda$, as desired.

Now, it remains for us to show that $\sB$ is an elementary substructure of $\sA$, clearly being a substructure. We wish to apply \Cref{TV test}, so take a formula $\varphi$, a free variable $x$ in $\varphi$ (the other variables of $\varphi$ being $y_{1}$ through $y_{n}$, of sorts, respectively, $\s_{1}$ through $\s_{n}$), and suppose that $(\sA, \nu)\vDash\Exists{x}\varphi$. Because of \Cref{GDNF}, we know we can write $\varphi$ as $\bigvee_{i=1}^{p}\bigwedge_{j=1}^{q}\varphi_{j}^{i}$, where $\varphi_{j}^{i}$ is a $\Sigma_{\lambda_{j}}$-formula: without loss of generality, suppose $x$ is of sort $\s$ in $\lambda_{q}$; then 
\[\Exists{x}\varphi=\bigvee_{i=1}^{p}\Exists{x}\varphi^{i}_{q}\wedge\bigwedge_{j=1}^{q-1}\varphi_{j}^{i},\]
and thus $(\sA,\nu)\vDash\Exists{x}\varphi^{i}_{q}\wedge\bigwedge_{j=1}^{q-1}\varphi_{j}^{i}$ for some $1\leq i\leq p$. Then the function $f^{x}_{\varphi^{i}_{q}}$, calculated on $\nu(y_{1})\in A_{\s}^{m_{1}}, \ldots , \nu(y_{n})\in A_{\s}^{m_{n}}$, returns a witness for $\Exists{x}\varphi^{i}_{q}$ in $A_{\s}^{\max\{m_{1},\ldots, m_{n}\}+1}$, a subset of $\sB$, which finishes our proof.\qed
\end{proof}

\section{Proof of \Cref{upwardssplit}}
\upwardsplitLS*
\begin{proof}
    For each $\s\in\S_{\Sigma}$, let $P_\s$ be a set of new constants of sort $\s$, where $|P_\s|$ has cardinality $|\s^{\sA}|$. For each $\lambda \in \Lambda$ and $\s\in \Sasa{\geq}(\lambda)$, let $Q_\s$ be a set of new constants of sort $\s$, where $|Q_\s|$ has cardinality $\kappa_\lambda$. Let $\Sigma_P$ be the signature obtained by adding the sets $P_\s$ to $\Sigma$, and let $\Sigma_Q$ be the signature obtained by adding the sets $Q_\s$ to $\Sigma_P$. We extend $\sA$ into a $\Sigma_{P}$-structure $\shA$ by interpreting the sorts, functions, and predicates in $\Sigma$ in the same way as in $\sA$, and defining $b_{\s}^{\shA}$, for $b_{\s}\in P_{\s}$, so that the mapping $b_{\s}\in P_{\s}\mapsto b_{\s}^{\shA}\in \s^{\shA}$ is bijective.

    Now, let $\Gamma$ be the set of all $\Sigma_{P}$-sentences satisfied by $\shA$, and let
    \[
        \overline{\Gamma}=\Gamma\cup\{\neg(c_{\s}=d_{\s}) : \text{$c_{\s}, d_{\s}\in Q_{\s}$, with $c_{\s}\neq d_{\s}$, for all $\s\in \Sasa{\geq}$}\}.
    \]
    By \Cref{compact}, $\overline{\Gamma}$ is consistent. Let $\sC$ be a model of $\overline{\Gamma}$. Then, for every $\s\in \Sasa{<}$, we have $|\s^\sC| = |\s^\sA|$, and for every $\s\in \Sasa{\geq}(\lambda)$, we have $|\s^\sC| \geq \kappa_\lambda$. 

    By \Cref{downwardssplit}, there is an elementary substructure $\sB$ of $\sC$ with $|\s^\sB| = \kappa_\lambda$ for every $\s\in \Sasa{\geq}(\lambda)$ and $\s^{\sB}=\s^{\sC}$ for each $\s\in \Sasa{<}$. Since $\sB \vDash \Gamma$, $\sA$ is isomorphic to an elementary substructure of $\sB$, where the isomorphism is given by $b_{\s}^{\shA}\mapsto b_{\s}^{\sB}$ for $b_\s \in P_\s$. Identifying these elements of $\sB$ with the corresponding elements of $\sA$ completes the proof.\qed
\end{proof}

\section{Proof of \Cref{middlesplit}}
\middlesplitLS*
\begin{proof}
First, apply \Cref{downwardssplit} to $\sA$ with the cardinals $\theta_\lambda = \max\{\aleph_{0}, |\Sigma_{\lambda}|\}$ to get a structure $\sC$ elementarily equivalent to $\sA$ satisfying $|\s^{\sC}|=\max\{\aleph_{0}, |\Sigma_{\lambda}|\}$ for every $\s\in \Sasa{\geq}(\lambda)$ and $\s^{\sC} = \s^{\sA}$ for every $\s\in \Sasa{<}$. Then, apply \Cref{upwardssplit} to $\sC$ with the cardinals $\kappa_\lambda$ to get a structure $\sB$ elementarily equivalent to $\sA$ with $|\s^{\sB}|=\kappa_\lambda$ for every $\s\in \Sasa{\geq}(\lambda)$ and $\s^{\sB} = \s^{\sA}$ for every $\s\in \Sasa{<}$.\qed
\end{proof}

\section{Proof of \Cref{losvaughttranslation}}
\losvaughttranslation*

\begin{proof}
    The translation of $\T$ into a single-sorted $\Sigma^\dag$-theory $\T^\dag$ is described in \cite{quineconj}. We informally recapitulate the translation here. First, we may assume without loss of generality that $\Sigma$ has no function symbols, since function symbols can be eliminated in favor of suitably axiomatized predicate symbols. Let $\Sigma^\dag$ have the components
    \begin{align*}
        \S_{\Sigma^\dag} &= \{\s^\dag\} \\
        \F_{\Sigma^\dag} &= \emptyset \\
        \P_{\Sigma^\dag} &= \{P_\s : \s \in \S_{\Sigma}\} \cup \P_{\Sigma}.
    \end{align*}
    Now, a $\Sigma$-formula $\varphi$ can be translated to $\Sigma^\dag$-formula $\varphi^\dag$ by using the predicates $P_\s$ to relativize quantifiers to their respective sorts, as in the proof of \Cref{Lowenheim}.
    Then, let
    \[
        \T^\dag = \T^\dag_1 \cup \T^\dag_2 \cup \T^\dag_3,
    \]
    where
    \begin{itemize}
        \item $\T^\dag_1$ is the set of translated axioms of $\T$,
        \item $\T^\dag_2$ is a set of sentences asserting that every element belongs to exactly one sort (this is where we use the assumption that $\Sigma$ has finitely many sorts), and every sort is nonempty, and
        \item $\T^\dag_3$ is a set of sentences asserting that predicates are true only if they are applied to elements of the appropriate arity.
    \end{itemize}

    Now, there is a one-to-one correspondence between models of $\T$ and models of $\T^\dag$, such that a $\T$-model $\sA$ corresponds to a $\T^\dag$-model $\sB$ with $|{\s^\dag}^\sB| = \sum_{\s \in \S_{\Sigma}} |\s^\sA|$. Indeed, given a $\T$-model $\sA$, construct a $\T^\dag$-model $\sB$ by letting
    \begin{align*}
        {\s^\dag}^\sB &= \bigcup_{\s \in \S_{\Sigma}} \s^\sA \\
        P_\s^\sB &= \s^\sA \quad\text{for each} \ \s \in \S_{\Sigma} \\
        P^\sB &= P^\sA \quad\text{for each} \ P \in \P_{\Sigma}.
    \end{align*}
    Conversely, given a $\T^\dag$-model $\sB$, construct a $\T$-model $\sA$ by letting
    \begin{align*}
        \s^\sA &= P_\s^\sB \quad\text{for each} \ \s \in \S_{\Sigma} \\
        P^\sA &= P^\sB \quad\text{for each} \ P \in \P_{\Sigma}.
    \end{align*}

    Using this correspondence, we see that $\T^\dag$ is a $\Sigma^\dag$-theory all of whose models are infinite, and that $\T^\dag$ is $\kappa$-categorical. By \Cref{losvaughtunsorted}, $\T^\dag$ is complete. Thus, $\vdash_{\T^\dag} \varphi^\dag$ or $\vdash_{\T^\dag} \neg \varphi^\dag$ for every $\Sigma$-sentence $\varphi$. Hence, $\vdash_{\T} \varphi$ or $\vdash_{\T} \neg \varphi$, so $\T$ is complete.\qed
\end{proof}

\section{Proof of \Cref{losvaught}}
\losvaught*

\begin{proof}
    Suppose $\T$ is not complete. Then, for some sentence $\varphi$, the theories $\T_0 = \T \cup \{\varphi\}$ and $\T_1 = \T \cup \{\neg \varphi\}$ are consistent. Let $\sA_0$ and $\sA_1$ be models of $\T_0$ and $\T_1$ respectively. Since $\sA_0$ and $\sA_1$ are also models of $\T$, both are strongly infinite. By \Cref{middlesplit}, there are models $\sA'_0$ and $\sA'_1$ that are elementarily equivalent to $\sA_0$ and $\sA_1$ respectively such that $|\s^{\sA'_0}| = |\s^{\sA'_1}| = \kappa(\s)$ for all $\s\in \S_{\Sigma}$. Since $\sA'_0 \vDash \varphi$ and $\sA'_1 \vDash \neg \varphi$, the models $\sA'_0$ and $\sA'_1$ are not isomorphic, contradicting the assumption that $\T$ is $\kappa$-categorical.\qed
\end{proof}

\section{Illustrations}

\begin{figure}
\begin{center}
\adjustbox{max width=0.99\textwidth}{
\setlength\extrarowheight{-5pt}
\begin{tabular}{l  b{8cm} | b{8cm}  l}

& \begin{center}\underline{Downwards}\end{center} & \begin{center}\underline{Upwards} \end{center}\\

\raisebox{2\height }{\rotatebox[origin=c]{90}{Translation}} & 
\begin{center}
\begin{tikzpicture}[scale=0.95, every mark/.append style={draw=white}, mark size=2.4pt]
\begin{axis}[
    xmin=-0.1, xmax=1.3,
    ymin=0, ymax=1,
    axis lines=center,
    axis on top=true,
    axis y line=none,
    domain=0:1,
    clip=false,
    xtick={0, 0.1, 0.4, 0.6, 0.9, 1.2},
    xticklabels={$\s_{1}$, $\s_{2}$, $\s_{n}$, $\s^{\prime}_{1}$, $\s^{\prime}_{i}$,$\s^{\prime}_{m}$},
    extra x ticks={0.25, 0.75, 1.05},
    extra x tick style={tick style={draw=none}},
    extra x tick labels={$\cdots$, $\cdots$, $\cdots$}
    ]
    \addplot[mark=none, black, thick] coordinates {(-0.1,0.25) (1.3,0.25)};
    \addplot[mark=none, black, thick] coordinates {(-0.1,0.3475) (1.3,0.3475)};
    \node at (axis cs:-0.1, 0.215){$\aleph_{0}$};
    \node at (axis cs:-0.1, 0.3825){$\kappa$};
    \addplot[mark=none, black, thick, dotted] coordinates {(0.0,0) (0.0,0.2)};
    \addplot[mark=none, black, thick, dotted] coordinates {(0.1,0) (0.1,0.075)};
    \addplot[mark=none, black, thick, dotted] coordinates {(0.9,0) (0.9,0.3475)};
    \addplot[mark=none, black, thick, dotted] coordinates {(1.2,0) (1.2,0.29)};
    \addplot[mark=none, red, thick, dotted] coordinates {(1.2,0.29) (1.2,0.475)};
    \addplot[mark=none, red, thick, dotted] coordinates {(0.9,0.3475) (0.9,0.4)};%here
    \addplot[mark=none, black, thick, dotted] coordinates {(0.4,0) (0.4,0.125)};
    \addplot[mark=none, red, thick, dotted] coordinates {(0.6,0.315) (0.6,0.375)};
    \addplot[mark=none, black, thick, dotted] coordinates {(0.6,0) (0.6,0.315)};
    \addplot[only marks] 
table {
 0.6 0.315
 0.9 0.3475
 1.2 0.29
};
\addplot[only marks, rotated halfcircle=-90]
table {
 0.0 0.2
  0.1 0.075
 0.4 0.125
};
\addplot[only marks, mark=*,mark options={red}] 
table {
 0.6 0.375
 0.9 0.4
 1.2 0.475
};
\end{axis}
\end{tikzpicture}\\
(a)~\Cref{Lowenheim}
\vspace{-3mm}
\end{center} 
&
\begin{center}
\begin{tikzpicture}[scale=0.95, every mark/.append style={draw=white}, mark size=2.4pt]
\begin{axis}[
    xmin=-0.1, xmax=1.3,
    ymin=0, ymax=1,
    axis lines=center,
    axis on top=true,
    axis y line=none,
    domain=0:1,
    clip=false,
    xtick={0, 0.1, 0.4, 0.6, 0.9, 1.2},
    xticklabels={$\s_{1}$, $\s_{2}$, $\s_{n}$, $\s^{\prime}_{1}$, $\s^{\prime}_{i}$,$\s^{\prime}_{m}$},
    extra x ticks={0.25, 0.75, 1.05},
    extra x tick style={tick style={draw=none}},
    extra x tick labels={$\cdots$, $\cdots$, $\cdots$}
    ]
    \addplot[mark=none, black, thick] coordinates {(-0.1,0.25) (1.3,0.25)};
    \addplot[mark=none, black, thick] coordinates {(-0.1,0.53) (1.3,0.53)};
    \node at (axis cs:-0.1, 0.285){$\aleph_{0}$};
    \node at (axis cs:-0.1, 0.495){$\kappa$};
    \addplot[mark=none, black, thick, dotted] coordinates {(0.0,0) (0.0,0.2)};
    \addplot[mark=none, black, thick, dotted] coordinates {(0.1,0) (0.1,0.075)};
    \addplot[mark=none, black, thick, dotted] coordinates {(0.9,0) (0.9,0.53)};
    \addplot[mark=none, black, thick, dotted] coordinates {(1.2,0) (1.2,0.505)};
    \addplot[mark=none, black, thick, dotted] coordinates {(0.4,0) (0.4,0.125)};
    \addplot[mark=none, black, thick, dotted] coordinates {(0.6,0) (0.6,0.43)};
    \addplot[only marks] 
table {
 0.6 0.43
 0.9 0.53
 1.2 0.505
};
\addplot[only marks, rotated halfcircle=-90]
table {
 0.0 0.2
  0.1 0.075
 0.4 0.125
};
\addplot[only marks, mark=*,mark options={red}] 
table {
 0.6 0.375
 0.9 0.4
 1.2 0.475
};
\end{axis}
\end{tikzpicture}\\
(d)~\Cref{translated ULS}
\vspace{-3mm}
\end{center}
& \raisebox{2\height }{\rotatebox[origin=c]{90}{Translation}}
\\&&&\\&&&\\

\raisebox{2.5\height }{\rotatebox[origin=c]{90}{General}} & \begin{center}
\begin{tikzpicture}[scale=0.95, every mark/.append style={draw=white}, mark size=2.4pt]
\begin{axis}[
    xmin=-0.1, xmax=1.3,
    ymin=0, ymax=1,
    axis lines=center,
    axis on top=true,
    axis y line=none,
    domain=0:1,
    clip=false,
    xtick={0, 0.1, 0.4, 0.6, 0.9, 1.2},
    xticklabels={$\s_{1}$, $\s_{2}$, $\s_{n}$, $\s^{\prime}_{1}$, $\s^{\prime}_{i}$,$\s^{\prime}_{m}$},
    extra x ticks={0.25, 0.75, 1.05},
    extra x tick style={tick style={draw=none}},
    extra x tick labels={$\cdots$, $\cdots$, $\cdots$}
    ]
    \addplot[mark=none, black, thick] coordinates {(-0.1,0.25) (1.3,0.25)};
    \addplot[mark=none, black, thick] coordinates {(-0.1,0.3075) (1.3,0.3075)};
    \node at (axis cs:-0.1, 0.215){$\aleph_{0}$};
    \node at (axis cs:-0.1, 0.3425){$\kappa$};
    \addplot[mark=none, black, thick, dotted] coordinates {(0.0,0) (0.0,0.2)};
    \addplot[mark=none, black, thick, dotted] coordinates {(0.1,0) (0.1,0.075)};
    \addplot[mark=none, black, thick, dotted] coordinates {(0.9,0) (0.9,0.3075)};
    \addplot[mark=none, black, thick, dotted] coordinates {(1.2,0) (1.2,0.3075)};
    \addplot[mark=none, red, thick, dotted] coordinates {(1.2,0.3075) (1.2,0.475)};
    \addplot[mark=none, red, thick, dotted] coordinates {(0.9,0.3075) (0.9,0.4)};%here
    \addplot[mark=none, black, thick, dotted] coordinates {(0.4,0) (0.4,0.125)};
    \addplot[mark=none, red, thick, dotted] coordinates {(0.6,0.3075) (0.6,0.375)};
    \addplot[mark=none, black, thick, dotted] coordinates {(0.6,0) (0.6,0.3075)};
    \addplot[only marks] 
table {
 0.6 0.3075
 0.9 0.3075
 1.2 0.3075
};
\addplot[only marks, rotated halfcircle=-90]
table {
 0.0 0.2
  0.1 0.075
 0.4 0.125
};
\addplot[only marks, mark=*,mark options={red}] 
table {
 0.6 0.375
 0.9 0.4
 1.2 0.475
};
\end{axis}
\end{tikzpicture}\\
(b)~\Cref{generalized LST}
\vspace{-3mm}
\end{center} 
&\begin{center}
\begin{tikzpicture}[scale=0.95, every mark/.append style={draw=white}, mark size=2.4pt]
\begin{axis}[
    xmin=-0.1, xmax=1.3,
    ymin=0, ymax=1,
    axis lines=center,
    axis on top=true,
    axis y line=none,
    domain=0:1,
    xtick={0, 0.1, 0.4, 0.6, 0.9, 1.2},
    xticklabels={$\s_{1}$, $\s_{2}$, $\s_{n}$, $\s^{\prime}_{1}$, $\s^{\prime}_{i}$,$\s^{\prime}_{m}$},
    extra x ticks={0.25, 0.75, 1.05},
    extra x tick style={tick style={draw=none}},
    extra x tick labels={$\cdots$, $\cdots$, $\cdots$}
    ]
    \addplot[mark=none, black, thick] coordinates {(-0.1,0.25) (1.3,0.25)};
    \node at (axis cs:-0.05, 0.29){$\aleph_{0}$};
    \node at (axis cs:-0.05, 0.495){$\kappa$};
    \addplot[mark=none, black, thick, dotted] coordinates {(0.0,0) (0.0,0.2)};
    \addplot[mark=none, black, thick, dotted] coordinates {(0.1,0) (0.1,0.075)};
    \addplot[mark=none, black, thick, dotted] coordinates {(0.9,0) (0.9,0.53)};
    \addplot[mark=none, black, thick, dotted] coordinates {(1.2,0) (1.2,0.53)};
    \addplot[mark=none, black, thick, dotted] coordinates {(0.4,0) (0.4,0.125)};
    \addplot[mark=none, black, thick, dotted] coordinates {(0.6,0) (0.6,0.53)};
    \addplot[mark=none, black, thick] coordinates {(-0.1,0.53) (1.3,0.53)};
    \addplot[only marks] 
table {
 0.6 0.53
 0.9 0.53
 1.2 0.53
};
\addplot[only marks, rotated halfcircle=-90]
table {
 0.0 0.2
  0.1 0.075
 0.4 0.125
};
\addplot[only marks, mark=*,mark options={red}] 
table {
 0.6 0.375
 0.9 0.4
 1.2 0.475
};
\end{axis}
\end{tikzpicture}\\
(e)~\Cref{LowenheimSkolemUpwards}
\vspace{-3mm}

\end{center}
& \raisebox{2.5\height }{\rotatebox[origin=c]{90}{General}} \\&&&\\&&&\\&&&\\

\raisebox{5\height }{\rotatebox[origin=c]{90}{Split}} & \begin{center}
\begin{tikzpicture}[scale=0.95, every mark/.append style={draw=white}, mark size=2.4pt]
\begin{axis}[
    xmin=-0.1, xmax=1.8,
    ymin=0, ymax=1,
    axis lines=center,
    axis on top=true,
    axis y line=none,
    domain=0:1,
    clip=false,
    xtick={0, 0.1, 0.4, 0.6, 0.7, 1.0, 1.3, 1.6, 1.7},
    xticklabels={$\s_{1}$, $\s_{2}$, $\s_{n}$, $\s^{\prime}_{1}$, $\s^{\prime\prime}_{1}$,$\s^{\prime}_{i}$,$\s^{\prime\prime}_{j}$,$\s^{\prime}_{m}$,$\s^{\prime\prime}_{m}$},
    extra x ticks={0.25, 0.85, 1.15, 1.45},
    extra x tick style={tick style={draw=none}},
    extra x tick labels={$\cdots$, $\cdots$, $\cdots$, $\cdots$}
    ]
    \addplot[mark=none, black, thick] coordinates {(-0.1,0.25) (1.8,0.25)};
    \addplot[mark=none, black, thick] coordinates {(-0.1,0.3075) (1.8,0.3075)};
    \addplot[mark=none, black, thick] coordinates {(-0.1,0.4) (1.8,0.4)};
    \node at (axis cs:-0.1, 0.215){$\aleph_{0}$};
    \node at (axis cs:-0.1, 0.3525){$\kappa^{\prime}$};
    \node at (axis cs:-0.1, 0.435){$\kappa^{\prime\prime}$};
    \addplot[mark=none, black, thick, dotted] coordinates {(0.0,0) (0.0,0.2)};
    \addplot[mark=none, black, thick, dotted] coordinates {(0.1,0) (0.1,0.075)};
    \addplot[mark=none, black, thick, dotted] coordinates {(1.0,0) (1.0,0.3075)};
    \addplot[mark=none, black, thick, dotted] coordinates {(1.6,0) (1.6,0.3075)};
    \addplot[mark=none, red, thick, dotted] coordinates {(1.6,0.3075) (1.6,0.475)};
    \addplot[mark=none, red, thick, dotted] coordinates {(1.0,0.3075) (1.0,0.4)};%here
    \addplot[mark=none, black, thick, dotted] coordinates {(0.4,0) (0.4,0.125)};
    \addplot[mark=none, red, thick, dotted] coordinates {(0.6,0.3075) (0.6,0.375)};
    \addplot[mark=none, black, thick, dotted] coordinates {(0.6,0) (0.6,0.3075)};
    \addplot[mark=none, black, thick, dotted] coordinates {(0.7,0) (0.7,0.4)};
    \addplot[mark=none, red, thick, dotted] coordinates {(0.7,0.4) (0.7,0.445)};
    \addplot[mark=none, black, thick, dotted] coordinates {(1.3,0) (1.3,0.4)};
    \addplot[mark=none, red, thick, dotted] coordinates {(1.3,0.4) (1.3,0.465)};
    \addplot[mark=none, black, thick, dotted] coordinates {(1.7,0) (1.7,0.4)};
    \addplot[mark=none, red, thick, dotted] coordinates {(1.7,0.4) (1.7,0.545)};
    \addplot[only marks] 
table {
 0.6 0.3075
 0.7 0.4
 1.0 0.3075
 1.3 0.4
 1.6 0.3075
 1.7 0.4
};
\addplot[only marks, rotated halfcircle=-90]
table {
 0.0 0.2
  0.1 0.075
 0.4 0.125
};
\addplot[only marks, mark=*,mark options={red}] 
table {
 0.6 0.375
 0.7 0.445
 1.0 0.4
 1.3 0.465
 1.6 0.475
 1.7 0.545
};
\end{axis}
\end{tikzpicture} \\
(c)~\Cref{downwardssplit}
\vspace{-3mm}
\end{center} 
&\begin{center}
\begin{tikzpicture}[scale=0.95, every mark/.append style={draw=white}, mark size=2.4pt]
\begin{axis}[
    xmin=-0.1, xmax=1.8,
    ymin=0, ymax=1,
    axis lines=center,
    axis on top=true,
    axis y line=none,
    domain=0:1,
    clip=false,
    xtick={0, 0.1, 0.4, 0.6, 0.7, 1.0, 1.3, 1.6, 1.7},
    xticklabels={$\s_{1}$, $\s_{2}$, $\s_{n}$, $\s^{\prime}_{1}$, $\s^{\prime\prime}_{1}$,$\s^{\prime}_{i}$,$\s^{\prime\prime}_{j}$,$\s^{\prime}_{m}$,$\s^{\prime\prime}_{m}$},
    extra x ticks={0.25, 0.85, 1.15, 1.45},
    extra x tick style={tick style={draw=none}},
    extra x tick labels={$\cdots$, $\cdots$, $\cdots$, $\cdots$}
    ]
    \addplot[mark=none, black, thick] coordinates {(-0.1,0.25) (1.8,0.25)};
    \addplot[mark=none, black, thick] coordinates {(-0.1,0.51) (1.8,0.51)};
    \addplot[mark=none, black, thick] coordinates {(-0.1,0.59) (1.8,0.59)};
    \node at (axis cs:-0.1, 0.285){$\aleph_{0}$};
    \node at (axis cs:-0.1, 0.475){$\kappa^{\prime}$};
    \node at (axis cs:-0.1, 0.555){$\kappa^{\prime\prime}$};
    \addplot[mark=none, black, thick, dotted] coordinates {(0.0,0) (0.0,0.2)};
    \addplot[mark=none, black, thick, dotted] coordinates {(0.1,0) (0.1,0.075)};
    \addplot[mark=none, black, thick, dotted] coordinates {(1.0,0) (1.0,0.51)};
    \addplot[mark=none, black, thick, dotted] coordinates {(1.6,0) (1.6,0.51)};
    \addplot[mark=none, black, thick, dotted] coordinates {(0.4,0) (0.4,0.125)};
    \addplot[mark=none, black, thick, dotted] coordinates {(0.6,0) (0.6,0.51)};
    \addplot[mark=none, black, thick, dotted] coordinates {(0.7,0) (0.7,0.59)};
    \addplot[mark=none, black, thick, dotted] coordinates {(1.3,0) (1.3,0.59)};
    \addplot[mark=none, black, thick, dotted] coordinates {(1.7,0) (1.7,0.59)};
    \addplot[only marks] 
table {
 0.6 0.51
 0.7 0.59
 1.0 0.51
 1.3 0.59
 1.6 0.51
 1.7 0.59
};
\addplot[only marks, rotated halfcircle=-90]
table {
 0.0 0.2
  0.1 0.075
 0.4 0.125
};
\addplot[only marks, mark=*,mark options={red}] 
table {
 0.6 0.375
 0.7 0.445
 1.0 0.4
 1.3 0.465
 1.6 0.475
 1.7 0.545
};
\end{axis}
\end{tikzpicture}\\
(f)~\Cref{upwardssplit}
\vspace{-3mm}
\end{center} &
\raisebox{5\height }{\rotatebox[origin=c]{90}{Split}}
\end{tabular}}
\end{center}
    \caption{Downward and Upward  L{\"o}wenheim--Skolem Theorems.}
    \label{fig:diagrams}
\end{figure}
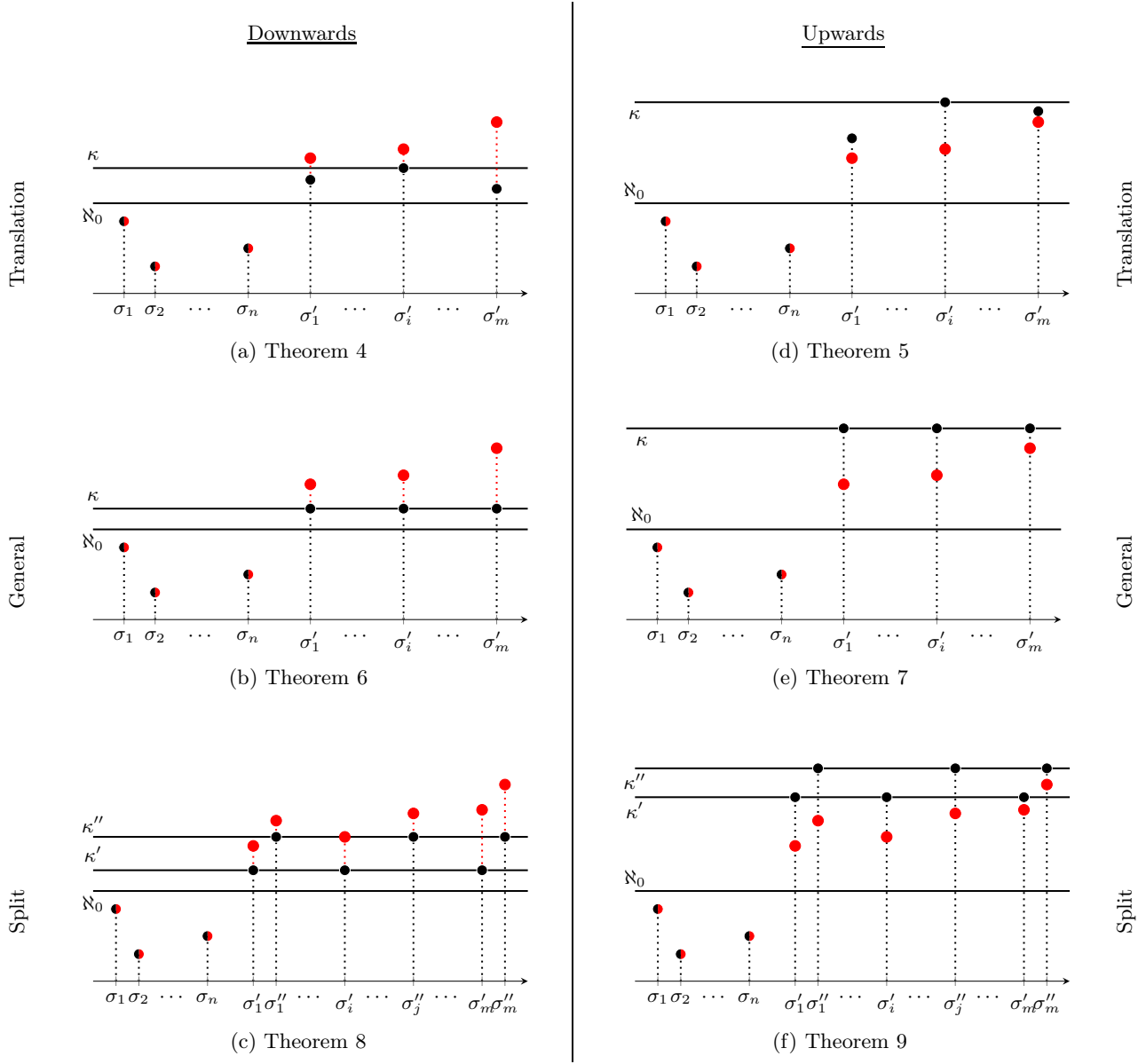

%shortcut: diagrams
In the body of the paper, we included \Cref{fig:illmid,fig:illmidsplit}, in order
to visualize the effect of the new combined L{\"o}wenheim--Skolem theorems.
We only included the diagrams for the combined theorems.
In this appendix, we include similar diagrams for the downward and upward theorems as well.

\begin{description}
    \item[Translation-based proofs] 
\Cref{Lowenheim,translated ULS} are proven using a translation-based approach.
We visualize their effects in,
 respectively, 
\Cref{fig:diagrams}.(a) and \Cref{fig:diagrams}.(d).
Notice  that $\s^{\prime}_{i}$ is the sort in the resulting structure that gets assigned the desired cardinality $\kappa$, although more than one such sort could exist; the final cardinalities for other sorts with infinite domains are bounded from below by $\aleph_{0}$ and from above by $\kappa$, but we do not have any further control over them.
\item[Direct proofs]
\Cref{generalized LST,LowenheimSkolemUpwards} 
are proven in a direct manner, by addapting the single-sorted classical proofs
to the many-sorted case.
We visualize \Cref{generalized LST,LowenheimSkolemUpwards} in 
\Cref{fig:diagrams}.(b) and \Cref{fig:diagrams}.(e), respectively.
With these new results, we are able to set the cardinalities
of all the infinite sorts, but only to the exact same cardinal.

\item[Direct proofs for split signatures]
% \Cref{downwardssplit,upwardssplit} are also proved
% using an adaptation of the single-sorted case, but the adaptation is more sophisticated,
% as it involves a new kind of normal forms.
We provide examples of applying \Cref{downwardssplit,upwardssplit} in, respectively, 
\Cref{fig:diagrams}.(c), \Cref{fig:diagrams}.(f).
We assume that our signature is split into $\Sigma_{\lambda_{1}}$ and $\Sigma_{\lambda_{2}}$, where $\Sasa{\geq}(\lambda_{1})=\{\s^{\prime}_{1},\ldots,\s^{\prime}_{m}\}$ and $\Sasa{\geq}(\lambda_{2})=\{\s^{\prime\prime}_{1},\ldots,\s^{\prime\prime}_{m}\}$ 
%(we choose $|\lambda_{1}|=|\lambda_{2}|$ to keep things simple, but in general, this may not hold). 
(the sorts with finite interpretations can belong to either $\lambda_1$ or $\lambda_2$).
Then, $\kappa^{\prime}$ is the cardinal associated with $\Sigma_{\lambda_{1}}$, and $\kappa^{\prime\prime}$ the cardinal associated with $\Sigma_{\lambda_{2}}$.
Thus, for split signatures, we are able
to choose a cardinality for each class of sorts.
\end{description}

Overall, going downward in either the left or right part of \Cref{fig:diagrams},
we see that the translation-based proofs only let us set
the maximal cardinality
(diagrams (a) and (d)), 
while the theorems proved using the direct approach allow us to set all sorts
to a single cardinality (diagrams (b) and (e)).
For split signatures, we can do even better, and have a dedicated
cardinality for each class of sorts (diagrams (c) and (f)).

\end{document}